\documentclass[12pt, reqno, a4paper]{amsart}
\pdfoutput=1

\usepackage[margin=1in]{geometry}
\usepackage{amssymb}
\usepackage{mathrsfs}
\usepackage[breaklinks=true]{hyperref}
\hypersetup{
    colorlinks = true,
    urlcolor = black,
    linkcolor= black,
    citecolor= black,
    filecolor= black,
}
\usepackage{graphicx}
\usepackage{lmodern}
\usepackage{mathtools}
\usepackage{tikz-cd}
\usepackage{enumitem}

\newcommand{\mono}{\rightarrowtail} 
\newcommand{\repi}{\twoheadrightarrow} 
\newcommand{\into}{\hookrightarrow} 

\newcommand{\Set}{{\bf Set}}
\newcommand{\CHLoc}{{\bf CHLoc}}
\newcommand{\Loc}{{\bf Loc}}
\newcommand{\DLat}{{\bf DLat}}
\newcommand{\BA}{{\bf BA}}
\newcommand{\Frm}{{\bf Frm}}
 \newcommand{\slat}{{\bf SLat}} 
 \newcommand{\KHaus}{{\bf CHSp}}
 
 \newcommand{\cat}[1]{\mathcal{#1}} 
 \newcommand{\K}{\cat{K}} 
 
\newcommand{\F}{\mathscr{S}} 
\DeclareMathOperator{\pts}{\mathbf{pt}} 
\renewcommand{\P}{\mathscr{P}} 
\DeclareMathOperator{\Idl}{\mathrm{Idl}} 
\DeclareMathOperator{\B}{\mathcal{B}} 
\DeclareMathOperator{\Sh}{\mathrm{Sh}} 
\DeclareMathOperator{\Int}{\mathcal{I}} 
\newcommand{\op}{\mathrm{op}}

\let\originalleft\left
\let\originalright\right
\renewcommand{\left}{\mathopen{}\mathclose\bgroup\originalleft}
\renewcommand{\right}{\aftergroup\egroup\originalright}

\DeclareMathOperator{\Subobjects}{\mathrm{Sub}}
\newcommand{\sub}{\Subobjects} 
\newcommand{\Sub}[1]{\Subobjects\left(#1\right)}
 
 \renewcommand{\phi}{\varphi}
\renewcommand{\epsilon}{\varepsilon}

\newcommand{\zero}{\mathbf{0}} 
\newcommand{\one}{\mathbf{1}} 

\newcommand{\down}{\mathop{\downarrow}}
\newcommand{\sleq}{\sqsubseteq} 
\newcommand{\swedge}{\sqcap} 
\newcommand{\svee}{\sqcup} 
\newcommand{\bigsvee}{\bigsqcup} 

\newcommand{\WEM}{\mathbf{WEM}} 

\theoremstyle{plain}
\newtheorem{theorem}{Theorem}[section]
\newtheorem{corollary}[theorem]{Corollary}
\newtheorem{lemma}[theorem]{Lemma}
\newtheorem{proposition}[theorem]{Proposition}

\theoremstyle{definition}
\newtheorem{definition}[theorem]{Definition}
\newtheorem{remark}[theorem]{Remark}

\newtheorem{example}[theorem]{Example}

\makeatletter
\renewcommand{\tocsection}[3]{%
  \indentlabel{\@ifnotempty{#2}{\bfseries\ignorespaces#1 #2\quad}}\bfseries#3}
\renewcommand{\tocsubsection}[3]{%
  \indentlabel{\@ifnotempty{#2}{\ignorespaces#1 #2\quad}}#3}

\newcommand\@dotsep{4.5}
\def\@tocline#1#2#3#4#5#6#7{\relax
  \ifnum #1>\c@tocdepth 
  \else
    \par \addpenalty\@secpenalty\addvspace{#2}%
    \begingroup \hyphenpenalty\@M
    \@ifempty{#4}{%
      \@tempdima\csname r@tocindent\number#1\endcsname\relax
    }{%
      \@tempdima#4\relax
    }%
    \parindent\z@ \leftskip#3\relax \advance\leftskip\@tempdima\relax
    \rightskip\@pnumwidth plus1em \parfillskip-\@pnumwidth
    #5\leavevmode\hskip-\@tempdima{#6}\nobreak
    \leaders\hbox{$\m@th\mkern \@dotsep mu\hbox{.}\mkern \@dotsep mu$}\hfill
    \nobreak
    \hbox to\@pnumwidth{\@tocpagenum{\ifnum#1=1\fi#7}}\par
    \nobreak
    \endgroup
  \fi}
\AtBeginDocument{%
\expandafter\renewcommand\csname r@tocindent0\endcsname{0pt}
}
\def\l@subsection{\@tocline{2}{0pt}{2.5pc}{5pc}{}}
\makeatother

\author{C\'elia Borlido}
\address{CMUC, Departamento de Matem\'atica, Universidade de Coimbra, 3001-501 Coimbra, Portugal}
\email{cborlido@mat.uc.pt}

\author{Panagis Karazeris}
\address{Department of Mathematics, University of Patras, 26504 Patras, Greece}
\email{pkarazer@upatras.gr}

\author{Luca Reggio}
\address{Department of Computer Science, University College London, 66--72 Gower Street, London WC1E 6EA, United Kingdom}
\email{l.reggio@ucl.ac.uk}

\author{Konstantinos Tsamis}
\address{Department of Mathematics, University of Patras, 26504 Patras, Greece}
\email{tsamiskonstantinos@gmail.com}

\title{Filtral pretoposes and compact Hausdorff locales}

\thanks{C.~Borlido was partially supported by the Centre for Mathematics of the University of Coimbra - UIDB/00324/2020, funded by the Portuguese Government through FCT/MCTES. L.~Reggio acknowledges support from the EPSRC grant EP/V040944/1. K.~Tsamis benefited from the Operational Programme \emph{Human Resources Development, Education and Lifelong Learning}, co-financed by Greece and the European Union (European Social Fund - ESF), in the context of the project \emph{Strengthening Human Resources Research Potential via Doctorate Research} (MIS-5000432).}

\begin{document}
\maketitle

\begin{abstract}
The category of compact Hausdorff locales is a pretopos which is filtral, meaning that every object is covered by one whose subobject lattice is isomorphic to the lattice of filters of complemented elements. We show that any filtral pretopos satisfying some mild additional conditions can be embedded into the category of compact Hausdorff locales. This result is valid in the internal logic of any topos. Assuming the principle of weak excluded middle and the existence of copowers of the terminal object in the pretopos, the image of the embedding contains all spatial compact Hausdorff locales.

The notion of filtrality was introduced by V.~Marra and L.~Reggio (Theory Appl.\ Categ., 2020) to characterise the category of compact Hausdorff spaces within the class of pretoposes. Our results can be regarded as a constructive extension of the aforementioned characterisation, avoiding reference to points. If the ambient logic is classical, i.e.\ it satisfies excluded middle, and the prime ideal theorem for Boolean algebras holds, we obtain as a corollary the characterisation of compact Hausdorff spaces in \emph{op.\ cit}.
\end{abstract}

\tableofcontents

\vspace{-1.5em}

\section{Introduction}
A purely category theoretic characterisation of a class of mathematical objects, when available, can be regarded as a way to identify the essential properties of such a class, thus paving the way for an axiomatic study of categories sharing similar properties. A prime example is Lawvere's Elementary Theory of the Category of Sets, see~\cite{Lawvere1964,Lawvere2005}, which was later adapted by Schlomiuk to capture the class of topological spaces~\cite{Schlomiuk1970}.
In a similar vein, Marra and Reggio~\cite{mr} recently characterised the category $\KHaus$ of compact Hausdorff spaces as the unique, up to (weak) equivalence, non-trivial well-pointed pretopos that is \emph{filtral} and admits set-indexed copowers of its terminal object (see Theorem~\ref{th:mr-main}). 

Given the important role played by compact Hausdorff \emph{locales} in the development of mathematics internally in a topos,\footnote{Recall that Banaschewski and Mulvey's version of Gelfand duality~\cite{BM2006}, valid internally in any topos, concerns the category of compact \emph{completely regular} locales, and compact Hausdorff locales coincide with compact completely regular ones only if we assume the axiom of dependent choice in the metatheory. The constructive Gelfand duality was extended to compact Hausdorff locales by Henry in~\cite{Henry2016}.} and in particular in connection with topos-theoretic foundations of notions of space in quantum mechanics~\cite{hls}, it is natural to ask whether a similar characterisation is available for the category $\CHLoc$ of compact Hausdorff~locales. 

We hasten to point out that, as is well known, if we work in classical logic and assume the axiom of choice, or even a weaker principle such as the prime ideal theorem for Boolean algebras, then all compact Hausdorff locales are spatial, i.e.\ have enough points. Thus, the categories $\CHLoc$ and $\KHaus$ are equivalent. In contrast, in this paper we work in the setting of \emph{constructive} locale theory: our results, unless otherwise specified, are valid in the logic of any topos (which is intuitionistic) and we do not use non-constructive principles. In this setting, there exist non-spatial compact Hausdorff locales. 
Recall that \emph{Boolean toposes} are those whose internal logic is classical, i.e.\ satisfies the principle of excluded middle stating that, for any proposition~$P$,
\[
P \vee \neg P
\]
holds. 
Equivalently, the object of truth values $\Omega$ (the subobject classifier of the topos) is an internal Boolean algebra. 
The existence of enough points for all compact Hausdorff locales implies that the ambient logic is classical, cf.\ Lemma~\ref{l:enough-points-lem}. Hence, any non-Boolean topos contains a non-spatial compact Hausdorff locale. 

From the standpoint of constructive locale theory, a major shortcoming of the approach in~\cite{mr} is its reliance on the use of points (i.e.\ global elements). On the other hand, it contributes a notion of \emph{filtrality} that turns out to be fundamental also in the point-free setting: a (coherent) category is filtral if it has enough objects whose lattices of subobjects are order-dual to Stone frames; see Section~\ref{ss:filtral-cats} for a precise definition. 

Constructively, one can show that $\CHLoc$ is a filtral pretopos. 
Conversely, our main results identify sufficient conditions on a filtral pretopos $\K$ such that:
\begin{itemize}
\item $\K$ admits an embedding into $\CHLoc$ (Theorem~\ref{t:Sub-pretopos-morphism});
\item if the principle of weak excluded middle $\neg P \vee \neg \neg P$ holds and $\K$ has copowers of its terminal object, the image of this embedding contains all spatial compact Hausdorff locales (Theorem~\ref{th:spatial-locales-covered}).
\end{itemize}
If the ambient logic is classical and compact Hausdorff locales have enough points, we can derive from Theorem~\ref{th:spatial-locales-covered} the characterisation of compact Hausdorff spaces given in~\cite{mr}. 
The problem of obtaining a satisfactory constructive characterisation of $\CHLoc$ remains open; we note that in Corollary~\ref{cor:essent-surj} we obtain a constructive characterisation of $\CHLoc$, but it involves quantifying over the class of Boolean algebras.

\vspace{0.5em}
In the remainder of this introductory section we offer a more detailed overview of the content of the article. In Section~\ref{s:preliminaries} we recall the necessary background concerning locales and pretoposes. 
In Section~\ref{s:sub-fun-into-CHLoc} we recall the concept of a filtral category and show that, for any such category $\K$, there is a natural functor $\F\colon \K\to \CHLoc$ which assigns to an object of $\K$ its lattice of subobjects with the dual order. To show that $\F$ lands in the category of compact Hausdorff locales, we prove that a closed quotient of a compact Hausdorff locale is Hausdorff, an observation which may be of independent interest.

Further properties of $\K$, including a refinement of the notion of filtrality called \emph{compatible filtrality}, which guarantee that the functor $\F$ is faithful and preserves finite limits, are investigated in Section~\ref{sec:lim}.
This is in stark contrast with the case of compact Hausdorff spaces, where the preservation of arbitrary limits follows easily from the assumption that the pretopos is well-pointed, combined with the fact that limits in the category of compact Hausdorff spaces are created by the underlying-set functor; cf.\ \cite[Lemma~4.7]{mr}. 

In Section~\ref{sec:main} we show that, if $\K$ is a filtral pretopos satisfying appropriate assumptions, the functor $\F\colon \K\to\CHLoc$ is a pretopos embedding, i.e.\ a fully faithful pretopos morphism (Theorem~\ref{t:Sub-pretopos-morphism}). While up to this point our methods are valid in the internal logic of a topos, in Section~\ref{s:essential image} we study the essential image of the embedding $\F$ assuming the principle of weak excluded middle, valid in de Morgan toposes. 
Theorem~\ref{th:spatial-locales-covered} states that, if $\K$ additionally has set-indexed copowers of its terminal object $\one$, the essential image of $\F$ contains all spatial compact Hausdorff locales and their closed sublocales. 

Finally, in Section~\ref{s:CH-spaces} we derive Marra and Reggio's characterisation of the category of compact Hausdorff spaces from Theorem~\ref{th:spatial-locales-covered}. To this end, we prove that the compatible filtrality property holds whenever $\one$ is an \emph{atom}, i.e.\ it has exactly two subobjects. 

\addtocontents{toc}{\protect\setcounter{tocdepth}{1}}
\subsection*{Notation} We write $\mono$ and $\repi$ to indicate that an arrow is a monomorphism or a regular epimorphism, respectively. Whenever they exist, the initial and terminal objects of a category are denoted, respectively, by $\zero$ and $\one$. Similarly, the least and greatest elements of a poset are denoted, respectively, by $0$ and $1$.
\addtocontents{toc}{\protect\setcounter{tocdepth}{2}}

\section{Preliminaries}\label{s:preliminaries}

\subsection{Frames and locales}\label{sec:locales}
  We recall some basic (constructively valid) facts about frames and locales that will be needed in the remainder of the article. Classical references include e.g.~\cite{jo,pp,ppt}; for a constructive treatment, see \cite[\S C1]{Elephant2}.
  
 A \emph{frame} is a complete lattice $L$ satisfying the infinite distributive law
 \[
 u\wedge \bigvee_{i\in I}{v_i} = \bigvee_{i\in I}{(u\wedge v_i)}
 \]
 for all subsets $\{u\}\cup\{v_i\mid i\in I\}\subseteq L$. It follows that, for each $u\in L$, the monotone map $u\wedge -\colon L\to L$ preserves all suprema and thus has a right adjoint $u \to -\colon L\to L$. In particular, every element $v\in L$ admits a \emph{pseudocomplement} $\neg v\coloneqq v\to 0$, the largest element of $L$ disjoint from $v$. 
 A \emph{frame homomorphism} is a map $L\to M$ between frames that preserves finite infima and arbitrary suprema (in particular, frame homomorphisms preserve $0$ and $1$). The category of frames and their homomorphisms is denoted by $\Frm$, and is complete and cocomplete; moreover, the underlying-set functor $\Frm\to\Set$ is right adjoint and thus preserves limits. 
 
The category $\Loc$ is defined as the opposite of $\Frm$; the objects of $\Loc$ are referred to as \emph{locales} and its morphisms as \emph{localic maps}. 
If $X$ is a locale, we write $OX$ to denote the same object, this time regarded as a frame; we often refer to $OX$ as the frame corresponding to $X$. Accordingly, the frame homomorphism corresponding to a localic map $f \colon X \to Y$ is denoted by 
\[
f^* \colon OY \to OX.
\] 
Since $f^*$ preserves (finite infima and) all suprema, it has a right adjoint $f_*$:
\[
f^*\dashv f_*\colon OX \to OY.
\] 

As with any adjunction between posets, we have $f^* \cdot f_* \cdot f^* = f^*$ and $f_* \cdot f^* \cdot f_* = f_*$. Thus, the following conditions are equivalent:
\begin{itemize}
\item $f^*$ is surjective;
\item $f_*$ is injective.
\end{itemize}
If any of the previous equivalent conditions is satisfied, we say that $f$ is a localic \emph{injection}. Similarly, $f$ is a localic \emph{surjection} if any of the following equivalent conditions is satisfied:
\begin{itemize}
\item $f^*$ is injective (equivalently, it reflects the order);
\item $f_*$ is surjective.
\end{itemize} 
This terminology suggests that properties of a localic map $f$ are reflected by the right adjoint $f_*$ to the corresponding frame homomorphism. In fact, this viewpoint turns out to be far-reaching and has been adopted systematically e.g.\ in~\cite{pp}. In the following, we shall occasionally exploit this perspective and refer to the right adjoint to a frame homomorphism as a ``localic map''.

\subsubsection{Sublocales}\label{s:sublocales}
Localic injections are precisely the extremal monomorphisms in $\Loc$, and the extremal subobjects in $\Loc$ are known as \emph{sublocales}.\footnote{Recall that a monomorphism $m$ is \emph{extremal} if whenever $m = h\cdot e$ with $e$ an epimorphism, then $e$ is an isomorphism. Extremal subobjects are equivalence classes of extremal monomorphisms with common codomain; cf.\ Section~\ref{sec:reg-coh-pretoposes} for the notion of (plain) subobject.}
Sublocales of a locale $X$ correspond in an order-reversing manner to \emph{nuclei} on $OX$, i.e.\ inflationary, idempotent maps $j \colon OX \to OX$ that preserve finite infima. In more detail, if a sublocale of $X$ is represented by a localic injection $i\colon Y\mono X$, the corresponding nucleus is the composite
\[\begin{tikzcd}
OX \arrow{r}{i^*} & OY \arrow{r}{i_*} & OX.
\end{tikzcd}\]
The set $NX$ of all nuclei on $OX$ is a frame under the pointwise order, hence the set of all sublocales of a locale is a \emph{coframe}---the order-dual notion of frame. 

\emph{Open sublocales} $U \mono X$ are those given by nuclei $\mathfrak{o}_u \coloneqq u \to -$, for some $u\in OX$, and \emph{closed sublocales} $C \mono X$ are those given by nuclei $\mathfrak{c}_v \coloneqq v\vee - $, for some $v\in OX$; every nucleus is a supremum of nuclei of the form $\mathfrak{o}_u \wedge \mathfrak{c}_v$. Any localic map $f \colon X \to Y$ induces a frame homomorphism $f^- \colon NY \to NX$ determined by the fact that it preserves both open and closed sublocales (more precisely, $f^- \mathfrak{o}_u = \mathfrak{o}_{f^* u}$ and $f^- \mathfrak{c}_v = \mathfrak{c}_{f^* v}$), while it must preserve suprema. Its right adjoint $f_+\colon NX \to NY$ is given by $f_+ j= f_* \cdot j \cdot f^*$. 

\subsubsection{Compact Hausdorff locales}
The notion of compactness for locales is a direct abstraction of the one for topological spaces, while the localic notion of Hausdorffness is based on the characterisation of Hausdorff spaces as those whose diagonal is closed in the product topology. Let us recall that a localic map $f\colon X \to Y$ is \emph{closed} if it satisfies the dual Frobenius law
\begin{equation}\label{eq:closed}
\forall u\in OX, \forall v\in OY, \ \ f_* (u \vee f^* v)= f_* u \vee v.
\end{equation}
With this terminology, a locale $X$ is
\begin{enumerate}[leftmargin=5em]
\item[\emph{compact}] if whenever a directed supremum $\bigvee_{i\in
    I} u_i $, with $\{u_i\mid i\in I\} \subseteq OX$, equals the top
  element $1\in OX$, then $u_i=1$ for some $i\in I$.
\item[\emph{Hausdorff}] if its diagonal $X \to X \times X$ is closed.\footnote{This notion is sometimes referred to as \emph{strongly Hausdorff} or \emph{I-Hausdorff} in the literature, cf.\ e.g.~\cite{is,jo} and~\cite{pp}, respectively.}
\end{enumerate}

The full subcategory of $\Loc$ defined by compact Hausdorff locales is denoted by $\CHLoc$. An important fact is that every localic map $f$ between compact Hausdorff locales is \emph{proper}, i.e.\ $f$ is closed and $f_*$ preserves directed suprema. More generally, every localic map from a compact locale to a Hausdorff one is proper, see \cite[Corollary~4.4]{v}.

This yields a useful criterion for a morphism in $\CHLoc$ to be a localic surjection. In fact, as an immediate consequence of eq.~\eqref{eq:closed}, a closed localic map $f$ is surjective whenever it is \emph{dense}, i.e.\ $f_*(0)=0$.
Therefore, 
\begin{lemma}\label{l:dense-implies-surjective}
Any dense localic map between compact Hausdorff locales is surjective.
\end{lemma} 

\begin{remark}
Although we shall not need this fact, we point out that a proper localic map is surjective precisely when it is dense, hence the converse implication in Lemma~\ref{l:dense-implies-surjective} holds as well.
\end{remark}

Throughout, we shall rely on the following characterisation of monomorphisms and regular epimorphisms in $\CHLoc$; for a proof, see e.g.\ \cite[\S 3.6]{t}.
\begin{lemma}\label{l:CHLoc-monos-repis}
The following statements hold for any morphism $f$ in $\CHLoc$:
\begin{enumerate}[label=\textup{(}\alph*\textup{)}]
\item\label{i:monos-chloc} $f$ is a monomorphism if, and only if, it is a localic injection.
\item\label{i:repis-chloc} $f$ is a regular epimorphism if, and only if, it is a localic surjection.
\end{enumerate}
\end{lemma}

\begin{remark}
We will see in Section~\ref{s:pretoposes} below that $\CHLoc$ is a pretopos and every epimorphism in a pretopos is regular; hence, all epimorphisms in $\CHLoc$ are regular and they coincide with the localic surjections.
\end{remark}

We next recall some further properties of locales that will be useful in the following: a locale $X$ is
\begin{enumerate}[leftmargin=5em]
\item[\emph{regular}] if every $u \in OX$ is the supremum of those $v \in OX$ that are \emph{well inside} it, i.e.\ such that $\neg v \vee u =1$. 
\item[\emph{normal}] if whenever $a\vee b=1 \in OX$, there are $u,v\in OX$ such that $a\vee u =1$, $b \vee v=1$ and $u\wedge v= 0$. 
\item[\emph{subfit}] if every open sublocale of $X$ is a join of closed ones.\footnote{Classically, a locale $X$ is said to be subfit if for all $u,v\in OX$ such that $u\not\leq v$ there exists $a\in OX$ satisfying $u\vee a =1\neq v\vee a$. The above definition, which is classically equivalent, is preferable from a constructive point of view because it avoids the use of negation.} In terms of nuclei, if for all $u\in OX$
\[\mathfrak{o}_{u}=\bigwedge \{\mathfrak{c}_{v}\mid v\in OX, \ \mathfrak{o}_{u}\leq \mathfrak{c}_{v}\}.\]
\end{enumerate}

\begin{proposition}\label{p:compact-Hausdorff-equiv}
Let $X$ be a compact locale. The following statements are equivalent:
\begin{enumerate}[label=\textup{(}\arabic*\textup{)}]
\item\label{i:Haus} $X$ is Hausdorff.
\item\label{i:reg} $X$ is regular.
\item\label{i:norm-subfit} $X$ is normal and subfit.
\end{enumerate}
\end{proposition}
\begin{proof}
Compact Hausdorff locales coincide precisely with the compact regular ones, see e.g.\ \cite[Proposition~4.5 and Theorem~6.4]{ppt}, hence~\ref{i:Haus} $\Leftrightarrow$ \ref{i:reg}. Further, every regular locale is subfit (cf.\ \cite[2.3]{is}) and every compact regular locale is normal (see e.g.\ \cite[VII.2.2]{pp}), thus \ref{i:reg} $\Rightarrow$ \ref{i:norm-subfit}. Finally, \ref{i:norm-subfit} $\Rightarrow$ \ref{i:reg} because every normal subfit locale is regular, see e.g.\ \cite[Proposition~4.4]{ppt}.
\end{proof}

\subsubsection{Frames of ideals}\label{s:frames-of-ideals}
Let $L$ be a (bounded) distributive lattice. Recall that an \emph{ideal} on $L$ is a subset $I\subseteq L$ such that 
\begin{enumerate}[label=\textup{(}\roman*\textup{)}]
\item $I$ is closed under finite suprema, i.e.\ $0\in I$ and $x\vee y\in I$ whenever $x,y\in I$;
\item $I$ is downward closed, i.e.\ for all $x,y\in L$ such that $x\leq y$, if $y\in I$ then $x\in I$.
\end{enumerate}
The set $\Idl(L)$ of all ideals on $L$, ordered by set-theoretic inclusion, is a frame. Further, if $h\colon L\to M$ is a lattice homomorphism between distributive lattices, then the map
\[
\Idl(h)\colon \Idl(L)\to\Idl(M), \ \ I \mapsto \down h[I]\coloneqq \{y\in M\mid \exists x\in I \text{ such that } y\leq h(x) \}
\]
is a frame homomorphism. If $\DLat$ denotes the category of (bounded) distributive lattices and lattice homomorphisms, these assignments yield a functor
\begin{equation}\label{eq:ideal-functor-DLat-Frm}
\Idl\colon \DLat \to \Frm
\end{equation}
which is left adjoint to the forgetful functor $\Frm\to\DLat$. The unit $L\to \Idl(L)$ of the adjunction sends an element $x\in L$ to the principal ideal $\down x$, and is clearly injective. Thus, if $L$ is a distributive lattice and $M$ is a frame, any lattice homomorphism $L\to M$ admits a unique extension to a frame homomorphism $\Idl(L)\to M$. 

Consider now the restriction of the functor in eq.~\eqref{eq:ideal-functor-DLat-Frm} to the full subcategory $\BA$ of $\DLat$ defined by Boolean algebras. The frames (isomorphic to one) of the form $\Idl(B)$, with $B$ a Boolean algebra, are known as \emph{Stone frames}. These can be characterised as the distributive lattices $L$ such that the map  
\begin{equation}\label{eq:phi-L}
 \phi_L\colon L \to \Idl(\B(L)), \ \ x\mapsto \{c\in \B(L)\mid c\leq x\}
\end{equation}
is an order isomorphism, where $\B(L)$ denotes the \emph{Boolean center} of $L$, i.e.\ the Boolean algebra of complemented elements of $L$. The locales corresponding to Stone frames are referred to as \emph{Stone locales}. 

\begin{remark}\label{rem:localic-Stone-duality}
Although we shall not need this fact, let us mention in passing that the category $\BA$ of Boolean algebras is dually equivalent to the full subcategory of $\Loc$ defined by Stone locales. This is a localic version of the celebrated Stone duality  between Boolean algebras and Stone spaces (i.e., zero-dimensional compact Hausdorff spaces)~\cite{Stone1936}.
\end{remark}

The inclusion functor $\CHLoc\into \Loc$ has a left adjoint $\beta \colon \Loc \to \CHLoc$, known as the \emph{localic Stone-\v{C}ech compactification} functor. When $X$ is a locale whose corresponding frame is a Boolean algebra $B$, the frame $O(\beta X)$ is isomorphic to $\Idl(B)$, see e.g.\ \cite[p.~310]{BM1980}. For an arbitrary locale $Y$, the frame $O(\beta Y)$ can be identified with a \emph{subframe} of $\Idl(OY)$, cf.\ \cite[Proposition~3]{BM1980}. 
It follows that $\CHLoc$ is a complete and cocomplete category: in fact, $\CHLoc$ is closed under limits in $\Loc$, and colimits in $\CHLoc$ can be computed by first taking the colimit in $\Loc$ and then applying the functor $\beta$.

\subsection{Regular and coherent categories, and pretoposes}\label{sec:reg-coh-pretoposes}
Fix an arbitrary category~$\K$. Given an object $X \in \K$, we denote by $\Sub{X}$ the poset of subobjects of~$X$. 
We will often identify a subobject of $X$ with any of its representatives $m\colon S \mono X$, and sometimes even with just the domain $S$ of such a representative. Recall that 
\[
(m_1\colon  S_1 \mono X) \leq (m_2\colon  S_2 \mono X)
\]
in $\Sub{X}$ if, and only if, there is a morphism $u\colon  S_1 \to S_2$ satisfying $m_2 \cdot u = m_1$. 

There is, a priori, no reason why $\Sub{X}$ should be a set, as opposed to a proper class.  Categories in which every object has only a set of subobjects are called \emph{well-powered}. For instance, the category $\Loc$ is not well-powered, see e.g.\ \cite[p.~70]{pp}. However, all the (generic) categories considered in this paper are assumed to be well-powered.

If $\K$ has pullbacks, each poset $\Sub{X}$ is a $\wedge$-semilattice with top element, where the meet of two subobjects is given by their pullback (recall that monomorphisms are stable under pullback), and the top element by the identity arrow $X \to X$. 
Moreover, every morphism $f\colon  X \to Y$ in $\K$ induces a $\wedge$-semilattice homomorphism
\[f^{-1}\colon  \Sub{Y} \to \Sub{X}\]
preserving the top element, which sends a subobject $m\colon S \mono Y$ to the pullback of $m$ along $f$.  In fact, this assignment yields a functor
\[\sub\colon \K^\op \to \slat\] 
into the category of $\wedge$-semilattices and $\wedge$-semilattice homomorphisms preserving the top elements.
As suggested by the notation, $f^{-1}$ may be thought of as the \emph{preimage} map associated with~$f$.

\begin{example}\label{ex:subobjects-CHLoc}
  In $\CHLoc$, the subobjects of a compact Hausdorff locale can be identified with its closed sublocales (see e.g.\ \cite[p.~92]{t}). Further, the preimage map $f^{-1}\colon  \Sub{Y} \to \Sub{X}$ corresponding to an arrow $f\colon X\to Y$ in $\CHLoc$ can be described, in terms of nuclei, as the order-dual of the map $f^-\colon NY\to NX$ (defined in Section~\ref{s:sublocales}) restricted to closed nuclei. In other words, for all $v\in OY$, $f^{-1}$ sends the closed sublocale of $Y$ corresponding to the nucleus~$\mathfrak{c}_v$ to the closed sublocale of $X$ corresponding to~$\mathfrak{c}_{f^* v}$.
\end{example}

\subsubsection{Regular and coherent categories}
We shall now focus on two classes of categories, namely regular and coherent categories, defined by properties of the corresponding functors of subobjects $\sub\colon \K^\op \to \slat$. Suppose that $\K$ is a category admitting pullbacks.
For every arrow $f\colon  X \to Y$ in $\K$, the preimage map $f^{-1}\colon \Sub{Y} \to \Sub{X}$ is right adjoint if, and only if, for every subobject $m\colon S \mono X$ there exists a smallest subobject of~$Y$ through which $f \cdot m$ factors. If such a subobject of $Y$ exists, we shall denote it by $f[S]$ and refer to it as the \emph{image of $S$} under $f$. Suppose for a moment that $f^{-1}$ is right adjoint. In that case, we shall write
\[
  f[-]\colon \Sub{X}\to \Sub{Y}
\]
for its left adjoint---the \emph{image} map associated with $f$. In particular, considering the identity $X\to X$, we see that $f$ admits an image factorisation
\begin{equation}\label{eq:factor-image}
  X\to f[X]\mono Y.
\end{equation}
We say that $f$ has a \emph{pullback-stable image} if, for any morphism $g\colon Z\to Y$, taking the pullback of diagram~\eqref{eq:factor-image} along $g$ yields the image factorisation of the pullback of~$f$ along~$g$.

\begin{definition}
A category is \emph{regular} if it has finite limits and every morphism has a pullback-stable image.
\end{definition}

Given an arrow $f\colon X\to Y$ in a regular category, the morphism $X\to f[X]$ in eq.~\eqref{eq:factor-image} is a regular epimorphism. In fact, every regular category admits a (regular epimorphism, monomorphism) factorisation system given by image factorisations.
The following properties of regular categories will be useful in the following; for a proof, see e.g.\ \cite[Propositions~3.8 and~3.9, p.~142]{BarrGrilletOsdol1971}.
\begin{lemma}\label{l:reg-epi-1-1}
  Let $\K$ be a regular category and let $f$ be a morphism in
  $\K$. 
  \begin{enumerate}[label=\textup{(}\alph*\textup{)}]
  \item\label{item:1} If $f$ is a monomorphism then $f[-]$ is
    injective.
  \item\label{item:2} If $f$ is a regular epimorphism then $f[-]$ is
    surjective.
  \end{enumerate}
\end{lemma}

We shall now look at the case where the functor of subobjects corestricts to the category $\DLat$ of distributive lattices.
\begin{definition}
  A \emph{coherent category} is a regular category in which every poset (in fact, $\wedge$-semilattice) of subobjects has finite joins and, for every morphism
  $f\colon X\to Y$, the preimage map
  \[f^{-1}\colon \Sub{Y}\to \Sub{X}\] preserves them.
\end{definition}

\begin{example}
The category $\CHLoc$ of compact Hausdorff locales is coherent (and even a pretopos, see Example~\ref{ex:CHLoc-pretopos} below). On the other hand, $\Loc$ is not regular, let alone coherent, as its regular epimorphisms are not stable under composition~\cite[Corollary~3.7]{Plewe2000}. 
\end{example}

For every object $X$ of a coherent category $\K$, its poset of subobjects $\Sub{X}$ is a distributive lattice (see e.g.\ \cite[Lemma~A.1.4.2]{el1}), and the preimage maps $f^{-1}$ are lattice homomorphisms. Therefore, the functor $\sub\colon \K^\op\to \slat$ corestricts to a functor
\[
\sub\colon  \K^\op \to \DLat.
\]

Every coherent category has an initial object $\zero$ and the latter is \emph{strict}, meaning that every morphism $X\to \zero$ is an isomorphism; for a proof, see e.g.\ \cite[Lemma~A.1.4.1]{el1}.

\subsubsection{Pretoposes}\label{s:pretoposes}
A central notion for this paper is that of pretopos, which can be regarded as a special type of coherent category. Recall that a category is said to be
\begin{enumerate}[leftmargin=5em]
\item[\emph{positive}] if finite coproducts exist and are \emph{disjoint}, i.e.\ the pullback of any coproduct diagram $X \rightarrow X + Y \leftarrow Y$ yields the initial object $\zero$.
\item[\emph{effective}] if it is regular and every internal equivalence relation is the kernel pair of its coequaliser.\footnote{The \emph{kernel pair} of a morphism $f$ is the pullback of $f$ along itself. Effective categories are also called \emph{Barr-exact}; for a more thorough treatment, see e.g.~\cite{BarrGrilletOsdol1971} or \cite[\S\S 2.5--2.6]{Borceux2}.}
\end{enumerate}

\begin{definition}
A \emph{pretopos} is a positive and effective coherent category.
\end{definition}  

Equivalently, a pretopos can be defined as an effective category with finite coproducts that is \emph{extensive}; recall that, in the presence of finite coproducts and pullbacks, a category is extensive precisely when binary coproducts are disjoint and \emph{universal}, i.e.\ the pullback of a coproduct diagram  
$X \rightarrow X + Y \leftarrow Y$ 
along any morphism yields a coproduct diagram (see \cite[Proposition~2.14]{CLW93}).

\begin{example}\label{ex:CHLoc-pretopos}
The category $\CHLoc$ of compact Hausdorff locales is a pretopos; see \cite[Theorem~3.6.3]{t} for the fact that $\CHLoc$ is regular, and \cite[Theorem~4.4]{KarazerisTsamis2021} for the remaining properties. Note that, in view of Lemma~\ref{l:CHLoc-monos-repis}, the factorisation system on $\CHLoc$ determined by image factorisations is given by (localic surjections, localic injections).
\end{example}

All epimorphisms in a pretopos are regular. In particular, pretoposes are \emph{balanced} categories, meaning that every arrow that is both monic and epic must be an isomorphism. See e.g.\ \cite[Corollary~A.1.4.9]{el1} for a proof of these statements.

In the remainder of this section, we recall some well known facts about Cartesian functors and pretopos morphisms. A \emph{Cartesian} functor is a functor $F\colon \cat{C}\to\cat{D}$ between finitely complete categories that preserves finite limits. In particular, since $F$ preserves monomorphisms, for every object $X\in \cat{C}$ there is a well-defined (monotone) map
\begin{equation}\label{eq:full-bij-on-subs}
\Sub{X} \to \Sub{FX}, \ \ (m\colon S\mono X) \mapsto (Fm \colon FS\mono FX).
\end{equation}
We shall say that $F$ is \emph{full on subobjects} (respectively, \emph{bijective on subobjects}) if the previous map is surjective (respectively, bijective) for all objects $X\in \cat{C}$.
Moreover, recall that a functor is \emph{conservative} if it reflects isomorphisms.

A \emph{pretopos morphism} is a Cartesian functor between pretoposes that preserves finite coproducts and coequalisers of internal equivalence relations.
For the following characterisation of pretopos morphisms, cf.\ \cite[Proposition~2.4.4]{ma}.
\begin{lemma}\label{l:pretopos-morphism-iff-coherent}
A Cartesian functor between pretoposes is a pretopos morphism precisely when it preserves finite coproducts and regular epimorphisms.
\end{lemma}

The following conditions are equivalent for any pretopos morphism $F\colon \cat{C}\to \cat{D}$:
\begin{enumerate}[label=\textup{(}\arabic*\textup{)}]
\item\label{i:faithful} $F$ is faithful;
\item\label{i:conservative} $F$ is conservative;
\item\label{i:inj-on-subs} for all objects $X\in\cat{C}$, the map in eq.~\eqref{eq:full-bij-on-subs} is injective.
\end{enumerate}
Just observe that, since faithful functors reflect epimorphisms and monomorphisms, a faithful functor defined on a balanced category is conservative. Hence, \ref{i:faithful} $\Rightarrow$ \ref{i:conservative}. 
For the implication \ref{i:conservative} $\Rightarrow$ \ref{i:inj-on-subs}, cf.\ e.g.\ \cite[p.~125]{mrey}, while \ref{i:inj-on-subs} $\Rightarrow$ \ref{i:faithful} follows by considering the equaliser of any pair of parallel maps in $\cat{C}$ identified by $F$.

Let us say that a pretopos morphism is an \emph{embedding} if it is fully faithful. Combining the previous characterisation of faithful pretopos morphisms with the fact that a conservative Cartesian functor is full whenever it is full on subobjects (cf.\ e.g.\ \cite[p.~262]{ma}), we obtain the following criterion for a pretopos morphism to be an embedding.
\begin{lemma}\label{l:pretopos-mor-full-faithful}
A pretopos morphism is an embedding provided it is bijective on subobjects.
\end{lemma}
With regard to essential surjectivity, let us say that a pretopos morphism
\[
F\colon \cat{C}\to \cat{D}
\]
\emph{covers} an object $d\in \cat{D}$ if there exist an object $c\in\cat{C}$ and a (regular) epimorphism $Fc\repi d$. Further, $F$ \emph{covers its codomain} if it covers each object of $\cat{D}$.
The next observation is a consequence of \cite[Lemma~7.1.7]{mrey}.
\begin{lemma}\label{l:cover-essential-image}
Let $F\colon \cat{C}\to\cat{D}$ be a pretopos morphism that is bijective on subobjects. If an object $d\in\cat{D}$ is covered by $F$ then it belongs to the essential image of $F$, i.e.\ there exist $c\in \cat{C}$ and an isomorphism $Fc\cong d$.
\end{lemma}

Recall that a functor is a \emph{weak equivalence} if it is full, faithful and essentially surjective; every equivalence (i.e., a functor admitting a quasi-inverse) is a weak equivalence, but the converse requires the axiom of choice for classes.
Combining Lemmas~\ref{l:pretopos-mor-full-faithful} and~\ref{l:cover-essential-image}, we see that a pretopos morphism is a weak equivalence precisely when it is an embedding that covers its codomain.

\section{Filtral categories and the subobject functor into $\CHLoc$}\label{s:sub-fun-into-CHLoc}
 
 \subsection{Filtral categories}\label{ss:filtral-cats}
Since the category of compact Hausdorff locales is coherent (and even a pretopos, see Example~\ref{ex:CHLoc-pretopos}), there is a functor
\[
\sub\colon\CHLoc^\op \to \DLat.
\]
Note that, for every $X\in\CHLoc$, $\Sub{X}^\op$ is isomorphic to the frame of closed nuclei on $X$ (cf.\ Example~\ref{ex:subobjects-CHLoc}), and the latter frame is isomorphic to $OX$ via the map $u \mapsto \mathfrak{c}_u$.
In fact, when regarded as a functor into the category of (compact Hausdorff) locales, the functor $\Sub{-}^\op$ is naturally isomorphic to the identity functor of $\CHLoc$, i.e.\ 
\[
\Sub{-}^\op\colon \CHLoc \to \CHLoc
\]
is an equivalence of categories.
  
We shall exploit the previous observation to tackle the problem of characterising, up to weak equivalence, the category of compact Hausdorff locales. To this end, we recall from~\cite{mr} the notion of filtral category.\footnote{The definition given here differs slightly from the original one, as in the present article we opted to work with ideals rather than filters; the two definitions are easily seen to be equivalent.} To start with, consider a (bounded) distributive lattice $L$ and denote by $\B(L)$ its Boolean center. Recall from eq.~\eqref{eq:phi-L} the monotone map
 \[
 \phi_L\colon L \to \Idl(\B(L)), \ \ \phi_L(u)\coloneqq \{c\in \B(L)\mid c\leq u\}.
 \]
The lattice $L$ is said to be \emph{filtral} if $\phi_L$ is an order isomorphism. By extension, we say that an object $X$ of a coherent category is \emph{filtral} if the lattice $\Sub{X}^\op$ is filtral. Equivalently, $X$ is filtral if, and only if, $\Sub{X}^\op$ is a Stone locale (cf.\ Section~\ref{s:frames-of-ideals}).
 
 \begin{definition}
 A category is \emph{filtral} if it is coherent and each of its objects is covered by a filtral one, i.e.\ for every object $X$ there exist a filtral object $S$ and a regular epimorphism $S \repi X$.
 \end{definition} 
 
 \begin{example}
 The category $\CHLoc$ is filtral; this amounts to the well-known fact that every compact Hausdorff locale is the localic image of a Stone locale, e.g.\ via its \emph{Gleason cover} (see~\cite{Johnstone1980} or~\cite[\S D4.6]{Elephant2}).
 \end{example}

The notion of filtrality for coherent categories was introduced in~\cite{mr} to obtain a characterisation of the category $\KHaus$ of compact Hausdorff spaces and continuous maps which we now recall. 
Let us say that a category $\cat{C}$ is \emph{non-trivial} if it admits two non-isomorphic objects (in Section~\ref{ss:non-triviality} below we shall consider a different notion of \emph{non-triviality} that is better suited for a constructive approach). If $\cat{C}$ has a terminal object then it is \emph{well-pointed} if, for all morphisms $f,g\colon X\to Y$ in $\cat{C}$, $f = g$ whenever $f\cdot p = g\cdot p$ for all morphisms $p\colon \one \to X$. 
\begin{theorem}[{\cite[Theorem~5.1]{mr}}]\label{th:mr-main}
Up to equivalence, $\KHaus$ is the unique non-trivial well-pointed pretopos that is filtral and has all set-indexed copowers of its terminal object.
\end{theorem}

\begin{remark}
The previous result was obtained in a classical setting, i.e.\ under the (implicit) assumptions that the ambient logic is classical and that the prime ideal theorem holds. Its proof constructs a \emph{weak} equivalence; the existence of an equivalence as in the statement above follows from an application of the axiom of choice for classes.
\end{remark}

It was shown in~\cite[Lemma~4.4]{mr} that, if
$\K$ is a filtral category, the subobject functor 
\[
\sub\colon  \K^\op \to \DLat
\] 
corestricts to the category of coframes and coframe homomorphisms. Hence, for every object $X \in \K$, the order-dual $\Sub{X}^\op$
of $\Sub{X}$ is a frame and we have a functor $\K^\op \to
\Frm$ sending an object $X \in \K$ to $\Sub{X}^\op$ and a
morphism $f\colon X \to Y$ to~$f^{-1}$, now seen as a frame homomorphism $\Sub{Y}^\op
\to \Sub{X}^\op$. When passing to the order-duals, the adjoint pair $f[-] \dashv f^{-1}\colon \Sub{Y} \to \Sub{X}$ yields an adjoint pair 
\[
f^{-1} \dashv f[-] \colon \Sub{X}^\op \to \Sub{Y}^\op,
\] 
thus we can regard $f[-] \colon \Sub{X}^\op \to \Sub{Y}^\op$ as a localic map. To sum up, 

\begin{lemma}\label{l:2} 
Let $\K$ be a filtral category. The assignments $X \mapsto \Sub{X}^\op$ and $f
  \mapsto f[-]$ define a functor $\K\to \Loc$ into the category of locales.
\end{lemma}

In the remainder of this section, we prove some basic properties of the functor $\K\to \Loc$ just defined, starting with the key fact that it factors through the inclusion $\CHLoc \into \Loc$ (Proposition~\ref{p:functor-to-CHLoc}). This relies on a localic analogue of the topological result stating that the closed quotient of a compact Hausdorff space is compact and Hausdorff. 

\subsection{Closed images of compact Hausdorff locales are compact Hausdorff}
It is well known that the localic image of a compact locale is compact. The next proposition shows that the image of a compact Hausdorff locale under a closed localic map is also Hausdorff.\footnote{As we were not able to find this result in the literature, we offer a proof which essentially amounts to showing that subfitness is preserved under closed localic surjections; classically, the latter fact was observed e.g.\ in \cite[Remark~4.4]{GGK2014}.} 

\begin{proposition}\label{p:closed-image-chaus}
If $f \colon Y \repi X$ is a closed localic surjection and $Y$ is compact Hausdorff, then so is $X$.
\end{proposition}

\begin{proof}
The localic image of a compact locale is compact, see e.g.\ \cite[VII.1.3]{pp}, hence $X$ is compact. By Proposition~\ref{p:compact-Hausdorff-equiv}, it suffices to prove that $X$ is normal and subfit whenever $Y$ is. It is well known that normality is preserved under closed localic surjections, see e.g.\ \cite[VII.1.6]{pp2}, hence it remains to show that $X$ is subfit whenever $Y$ is.

Let $j = \mathfrak{o}_u$ correspond to an open sublocale $U$ of $X$. Then the inverse image of $U$ under $f$, as a sublocale of $Y$, corresponds to the nucleus $f^- j = \mathfrak{o}_{f^* u}$. Since $Y$ is subfit, there is a set $\{v_i\mid i\in I\}\subseteq OY$ such that $\mathfrak{o}_{f^* u} = \bigwedge_{i\in I}{\mathfrak{c}_{v_i}}$ in the frame of nuclei on $OY$, i.e.\ the inverse image of $U$ is a join of closed sublocales of $Y$. Hence we have 
\[
f_+ f^- j = f_+ (\bigwedge_{i\in I}{\mathfrak{c}_{v_i}}) = \bigwedge_{i\in I} f_+ \mathfrak{c}_{v_i}.
\]
As $f$ is a closed localic map, $f_+ \mathfrak{c}_{v_i}=\mathfrak{c}_{f_* v_i}$ for all $i\in I$. Thus, the previous equation exhibits the sublocale of $X$ corresponding to $f_+ f^- j$ as a join of closed sublocales. We claim that $j = f_+ f^- j$. The inequality $j \leq f_+ f^- j$ holds for all nuclei (not just the open ones). On the other hand, for all $w\in OX$ we have
\begin{align*}
f_+ f^- j(w) &\leq j(w)  \\
\Longleftrightarrow \ f_* (f^* u \to f^*w) &\leq u \to w \\
\Longleftrightarrow \ f_* (f^* u \to f^*w) \wedge u &\leq w \\
\Longleftrightarrow \ f^*f_* (f^* u \to f^*w) \wedge f^*u &\leq f^*w \\
\Longleftrightarrow \ f^*f_* (f^* u \to f^*w) &\leq  f^*u \to f^*w,
\end{align*}
which is always satisfied; note that the third equivalence holds because $f$ is a localic surjection and so the corresponding frame homomorphism $f^*$ reflects the order.

This shows that the open sublocale $U$ corresponding to the nucleus $j$ is a join of closed sublocales, and so $X$ is subfit.
\end{proof}

The following consequence of Proposition~\ref{p:closed-image-chaus}, although not needed in the present work, may be of independent interest. Recall that a localic map $Z\to X$ is \emph{separated} if the restricted diagonal $Z \to Z\times _X Z$ is closed.
\begin{corollary}
Consider a commutative triangle
\[\begin{tikzcd}
Z \arrow{rr}{f} \arrow{dr}[swap]{h} & & Y \arrow{dl}{g} \\
{} & X & {}
\end{tikzcd}\]
in the category of locales, with $f$ a closed localic surjection. If $h$ is proper and separated, then $g$ is also proper and separated. 
\end{corollary}

\begin{proof} 
We apply Proposition~\ref{p:closed-image-chaus} internally in the category $\Loc(\Sh(X))$ of locales in the sheaf topos $\Sh(X)$, using the equivalence 
\[\begin{tikzcd}
\Int\colon \Loc/{X}\arrow{r}{\simeq} & \Loc(\Sh(X)).
\end{tikzcd}\]
See e.g.\ \cite[Theorem~C1.6.3]{Elephant2}. Under this equivalence, a localic map $\ell\colon W\to X$ corresponds to the internal locale $\Int(\ell)$ in $\Sh(X)$ (called its \emph{internalization}) whose corresponding internal frame is given, as a sheaf, by
\[
\forall u\in OX, \ \ O\Int(\ell)(u)\coloneqq \{ v \in OW \mid v \leq \ell^* u \}.
\]
Further, a morphism $f$ in $\Loc/{X}$, from $h\colon Z\to X$ to $g\colon Y\to X$, is taken to the natural transformation $\Int (f)$ such that the component of the inverse image $\Int (f)^{*}$ at each $u \in OX$ is the appropriate restriction of $f^*$. On the other hand, the direct image $\Int (f)_*$ is given, for each $u \in OX$, by $\Int(f)_*(u)(x) = f_*x  \wedge g^* u$ where
\[
\Int (f)_*(u) \colon \Int (h)(u) = \{z \in OZ \mid  z \leq h^*u \} \to \Int (g)(u) = \{y \in OZ \mid y \leq g^*u \}.
\]

Suppose that $f$ is a closed localic surjection. 
We claim that the internal localic map $\Int(f)$ in $\Sh(X)$ is also a closed surjection.
Recall that localic surjections are precisely the epimorphisms in $\Loc$, and the forgetful functor $\Loc/{X} \to \Loc$ (preserves and) reflects colimits. Hence, $f$ is an epimorphism in $\Loc/{X}$. 
It follows that $\Int(f)$ is an epimorphism in $\Loc(\Sh(X))$, i.e.\ an internal localic surjection in $\Sh(X)$. As $f$ is closed, each component of the natural transformation $\Int(f)$ satisfies the dual Frobenius law; just observe that 
\begin{align*}
\Int(f)_*(u) (x \vee f^* y) &= f_* (x \vee f^* y) \wedge g^* u \\
&= (f_* x  \vee y) \wedge g^* u \tag*{$f$ is closed} \\ 
&= (f_* x  \wedge g^* u ) \vee (y \wedge g^* u )\\
&= \Int(f)_*(u) (x) \vee y. \tag*{$y \leq g^*u$}
\end{align*}
Since the dual Frobenius law is finitary equational, and the evaluation functors 
\[
\mathrm{ev}_u\colon \Sh(X)\to \Set
\] 
for $u\in OX$ are collectively faithful, $\Int(f)$ satisfies the dual Frobenius law as well. So,~$\Int(f)$ is a closed localic surjection internally in $\Sh(X)$.

Now, suppose that $h$ is proper and separated. A localic map $\ell\colon W\to X$ is proper and separated precisely when $\Int(\ell)$ is a compact Hausdorff internal locale in $\Sh(X)$, cf.\ e.g.\ \cite[p.~634]{Elephant2}. Therefore, applying Proposition~\ref{p:closed-image-chaus} to the internal localic map $\Int(f)$, we see that $g$ is also proper and separated. 
\end{proof}

\subsection{The functor \texorpdfstring{$\F\colon \K\to \CHLoc$}{S: K --> CHLoc}}

As promised, we show that the functor ${\K\to \Loc}$, defined as in Lemma~\ref{l:2}  for any filtral category $\K$, factors through the inclusion functor $\CHLoc \into \Loc$.
To avoid confusion, throughout we shall denote by $\sleq$ the order on~$\Sub{X}$
dual to~$\leq$. That is, for all $S,T\in\Sub{X}$,
\[
  S\leq T \ \text{in} \ \Sub{X} \ \Longleftrightarrow \ T \sleq S \
  \text{in} \ \Sub{X}^\op.
\]
Further, we shall denote infima and suprema in $\Sub{X}^\op$ by $\swedge$ and $\svee$, respectively, to distinguish them from infima and suprema in $\Sub{X}$.

\begin{proposition}\label{p:functor-to-CHLoc}
  Let $\K$ be a filtral category. The functor $\K\to \Loc$ in Lemma~\ref{l:2} corestricts to a functor $\F\colon \K\to \CHLoc$ into the category of compact Hausdorff locales. Moreover, $\F$ preserves monomorphisms and regular epimorphisms.
\end{proposition}

\begin{proof}
Recall that the image of a morphism $f\colon X\to Y$ in $\K$ under the functor $\K\to \Loc$ in Lemma~\ref{l:2} is the localic map
\[
f[-] \colon \Sub{X}^\op \to \Sub{Y}^\op
\]
whose left adjoint is the frame homomorphism $f^{-1}\colon \Sub{Y}^\op \to \Sub{X}^\op$.
The localic map $f[-]$ is closed because, in any regular category, the Frobenius law 
\[
f[U \wedge f^{-1} (V)]= f[U] \wedge V
\] 
holds at the level of subobject lattices, see e.g.\ \cite[Lemma A1.3.3]{el1}, and thus for the order-dual lattices of subobjects we have $f[U \svee f^{-1}(V)]=f[U] \svee V$.

We claim that, for each object $X$ of $\K$, the locale $\Sub{X}^\op$ is compact Hausdorff. Since~$\K$ is filtral, there exists a regular epimorphism $e \colon S \repi X$ with $S$ filtral. 
The localic map $e[-]\colon \Sub{S}^\op \to \Sub{X}^\op$ is closed by the first part of the proof and is a localic surjection by Lemma~\ref{l:reg-epi-1-1}\ref{item:2}. Now, $\Sub{S}^\op$ is a Stone locale (in particular, compact Hausdorff) because $S$ is filtral, hence the locale $\Sub{X}^\op$ is compact Hausdorff by Proposition~\ref{p:closed-image-chaus}.

This shows that the functor $\K\to \Loc$ corestricts to a functor $\K\to \CHLoc$. The fact that the latter preserves monomorphisms and regular epimorphisms is an immediate consequence of Lemmas~\ref{l:CHLoc-monos-repis} and~\ref{l:reg-epi-1-1}.
\end{proof}

The remainder of this paper is devoted to the study of the functor $\F\colon \K\to \CHLoc$, for $\K$ a filtral category possibly satisfying some extra conditions. Throughout, if $X$ is an object of $\K$, the notation $\Sub{X}^\op$ will refer to the frame of subobjects of $X$ partially ordered by $\sleq$. The locale corresponding to $\Sub{X}^\op$ will always be denoted by $\F(X)$.

We conclude this section by showing that the functor $\F$ is bijective on subobjects. In view of Proposition~\ref{p:functor-to-CHLoc} (more precisely, of the fact that $\F$ preserves monomorphisms), for every object $X$ of a filtral category $\K$ there is a well-defined map
\begin{align}\label{eq:subobj-comparison}
\Sub{X} &\to \Sub{\F(X)}, \\
(m\colon Y\mono X) &\mapsto (\F(m) = m[-]\colon \F(Y)\mono\F(X)) \nonumber
\end{align}
which is clearly monotone.

\begin{lemma}\label{l:full-on-subobjects}
For every object $X$ of a filtral category $\K$, the map in~\eqref{eq:subobj-comparison} is an order isomorphism. In particular, the functor $\F \colon \K \to \CHLoc$ is bijective on subobjects.
\end{lemma}

\begin{proof}
Recall that $\Sub{\F(X)}$ is order-dual to the lattice of closed nuclei on $\Sub{X}^\op$.

Now, denote by $\alpha$ the map in eq.~\eqref{eq:subobj-comparison} and consider a subobject $m\colon Y\mono X$. The frame homomorphism $\alpha(m)^*\colon \Sub{X}^\op\to\Sub{Y}^\op$ corresponding to the localic injection $\alpha(m)$ sends $S$ to $Y\svee S$, and its right adjoint $\alpha(m)_*$ is the inclusion map. Therefore, the sublocale of $\F(X)$ represented by $\alpha(m)$ corresponds to the closed nucleus $\mathfrak{c}_Y$ on the frame $\Sub{X}^\op$. In particular, it follows that $\alpha$ is surjective. To conclude that $\alpha$ is an order isomorphism it suffices to show that it reflects the order, i.e.\ that $\mathfrak{c}_{Y_2}\leq \mathfrak{c}_{Y_1}$ entails $Y_1\leq Y_2$ in $\Sub{X}$. In turn, this follows at once by evaluating the nuclei at $0$:
\[
Y_2 = \mathfrak{c}_{Y_2}(0) \sleq \mathfrak{c}_{Y_1}(0) = Y_1. \qedhere
\]
\end{proof}

\section{Faithfulness and preservation of finite limits}\label{sec:lim}
Let $\K$ be a filtral category. Our aim in this section consists in identifying conditions on~$\K$ ensuring that the functor $\F\colon \K\to\CHLoc$ is Cartesian; equivalently, that it preserves equalisers, binary products, and the terminal object. 

\subsection{Enough subobjects and preservation of equalisers} 
To start with, we shall postulate that $\F$ is faithful; this can be stated as a first-order axiom in the language of categories in the following way:
\begin{equation}\label{eq:enough-subobjects}
  \forall f,g \colon X \to Y \; ((\forall i \colon U \mono X, \; f [U]=g[U]) \Rightarrow f=g)).
\end{equation}

\begin{definition}
A regular category \emph{has enough subobjects} if it satisfies condition~\eqref{eq:enough-subobjects}.  
\end{definition}

\begin{example}
The category $\CHLoc$ has enough subobjects. Just observe that, for every localic map $f\colon X\to Y$ between compact Hausdorff locales, since a subobject $U$ of $X$ can be identified with a closed sublocale of $X$, which in turn corresponds to a nucleus $\mathfrak{c}_u$ for some $u\in OX$, the image $f[U]$ corresponds to the nucleus $f_+ \mathfrak{c}_{u}$. As $f$ is closed, the latter nucleus coincides with $\mathfrak{c}_{f_* u}$. Thus, given another localic map $g\colon X\to Y$, $f[U] = g[U]$ implies $\mathfrak{c}_{f_* u} = \mathfrak{c}_{g_* u}$, which in turn entails $f_* u = g_* u$. It follows that, if $f[U] = g[U]$ for all subobjects $U$ of $X$, then $f=g$.
\end{example}

The following easy observation will imply that $\F$ preserves equalisers:

\begin{lemma}\label{l:generic-lemma-equalisers}
Let $\cat{C},\cat{D}$ be categories with equalisers and let $F\colon \cat{C}\to\cat{D}$ be a faithful functor that preserves monomorphisms and such that, for all $a\in\cat{C}$, the canonical map $\Sub{a}\to \Sub{Fa}$ is an order isomorphism. Then $F$ preserves equalisers.
\end{lemma}

\begin{proof}
Consider an equaliser diagram in $\cat{C}$ as displayed below
\[\begin{tikzcd}
e \arrow[rightarrowtail]{r}{i} & a \arrow[yshift=3pt]{r}{f} \arrow[yshift=-3pt]{r}[swap]{g} & b
\end{tikzcd}\]
and let
\[\begin{tikzcd}
q \arrow[rightarrowtail]{r}{k} & Fa \arrow[yshift=3pt]{r}{Ff} \arrow[yshift=-3pt]{r}[swap]{Fg} & Fb
\end{tikzcd}\]
be an equaliser of $Ff$ and $Fg$ in $\cat{D}$. As $Fi$ equalises $Ff$ and $Fg$, the universal property of $k$ entails that $Fi \leq k$ in $\Sub{Fa}$. Since $\Sub{a}\to \Sub{Fa}$ is an order isomorphism, there exists a subobject $j\colon d\mono a$ such that $i\leq j$ in $\Sub{a}$ and $Fj=k$ in $\Sub{Fa}$. But then $Ff \cdot Fj = Fg\cdot Fj$ implies $f\cdot j = g \cdot j$ by faithfulness of $F$, and so by the universal property of $i$ we have $j \leq i$. Thus $i=j$ in $\Sub{a}$. We conclude that $Fi = k$ in $\Sub{Fa}$, showing that $Fi$ is an equaliser of $Ff$ and $Fg$. 
\end{proof}

We therefore deduce the next proposition from Lemma~\ref{l:generic-lemma-equalisers}, combined with Proposition~\ref{p:functor-to-CHLoc} and Lemma~\ref{l:full-on-subobjects}.

\begin{proposition}\label{p:preservation-of-equalizers} 
If $\K$ is a filtral category with enough subobjects, the functor $\F \colon \K \to \CHLoc$ preserves equalisers.
\end{proposition} 

\subsection{Binary products: compatible filtrality}
We now discuss the preservation of binary products by the functor $\F\colon\K\to\CHLoc$. To this end, we shall assume that~$\K$ satisfies an additional property, namely \emph{compatible filtrality}. In Proposition~\ref{p:one-atom-comp-filtral} below we shall see that this property holds whenever the terminal object of $\K$ has precisely two subobjects (however, as explained in Remark~\ref{rem:existence-K-em}, the latter sufficient criterion is only interesting in the classical setting).

\begin{definition}\label{def:comp-filtrality}
A category $\K$ is said to be \emph{compatibly filtral} if it is filtral and, for all filtral objects $S_1,S_2\in \K$, the canonical homomorphism of Boolean algebras
\begin{equation}\label{eq:BA-hom-inj-filt}
\B(\Sub{S_1}) + \B(\Sub{S_2}) \to \B(\Sub{S_1 \times S_2}) 
\end{equation}
is injective.
\end{definition}

\begin{example}
$\CHLoc$ is compatibly filtral. Let $B_1$ and $B_2$ be Boolean algebras, and let $S_1$ and $S_2$ be the Stone locales corresponding to the frames $\Idl(B_1)$ and $\Idl(B_2)$, respectively. It suffices to show that the canonical Boolean algebra homomorphism
\[
\B(\Sub{S_1}^\op) + \B(\Sub{S_2}^\op) \to \B(\Sub{S_1 \times S_2}^\op)
\]
is injective (cf.\ the proof of Lemma~\ref{l:comp-filtral-rephrasing} below). Since $\Sub{X}^\op\cong OX$ for all compact Hausdorff locales $X$, and the functor of ideals preserves coproducts, the previous homomorphism can be identified with
\[
\B(\Idl(B_1)) + \B(\Idl(B_2)) \to \B(\Idl(B_1 + B_2)).
\]
It is not difficult to see that $\B(\Idl(B))\cong B$ for every Boolean algebra $B$ (this can be regarded as part of the duality between Boolean algebras and Stone locales, cf.\ Remark~\ref{rem:localic-Stone-duality} and also~\cite[\S II.3]{jo}), hence the previous map is an isomorphism of Boolean algebras.
\end{example}

To clarify the relation between compatible filtrality and preservation of products by~$\F$, let us observe the following general fact:
\begin{lemma}\label{flat-functor-lemma}
  Let $F\colon \cat{C}\to \cat{D}$ be a functor with $\cat{C}$ finitely complete and $\cat{D}$ regular. Assume that $F$ preserves equalisers and, for any product diagram
\[\begin{tikzcd}
X & X \times Y \arrow{l}[swap]{\pi_1} \arrow{r}{\pi_2} & Y
\end{tikzcd}\]
in $\cat{C}$, the canonical arrow 
\[
\langle F \pi _1 , F \pi _2 \rangle \colon F(X\times Y) \to FX \times FY
\]
is a regular epimorphism. Then $F$ preserves binary products.
\end{lemma}

\begin{proof}
It suffices to show that the canonical morphism $\langle F \pi _1 , F \pi _2 \rangle$ in the statement is monic. To this end, consider a commutative diagram 
\[\begin{tikzcd}[column sep = 4em]
T \arrow[yshift= -3pt]{r}[swap]{v} \arrow[yshift= 3pt]{r}{u} & F(X \times Y) \arrow{r}{\langle F \pi _1 , F \pi _2 \rangle} & FX \times FY.
\end{tikzcd}\] 
We must prove that $u = v$. 
Let 
\[\begin{tikzcd}
X\times Y & (X\times Y) \times (X\times Y) \arrow{l}[swap]{p_1} \arrow{r}{p_2} & X\times Y
\end{tikzcd}\]
be a product diagram in $\cat{C}$ and consider the pullback of $\langle u,v \rangle$ along $\langle Fp_1 , Fp_2\rangle$:
\[\begin{tikzcd}[column sep = 5em]
T' \arrow{r}{w} \arrow[twoheadrightarrow]{d}[swap]{e} \arrow[dr, phantom, "\lrcorner", very near start, xshift = -0.4em, yshift = -0.6em] & F(X\times Y \times X \times Y) \arrow[twoheadrightarrow]{d}{\langle Fp_1 , F p_2 \rangle} \\
T \arrow{r}{\langle u,v \rangle} & F(X\times Y) \times F(X\times Y)
\end{tikzcd}\]
The morphism $\langle Fp_1 , Fp_2\rangle$ is a regular epimorphism by assumption, hence so is $e$.
By the commutativity of the previous square we have $Fp_1 \cdot w = u \cdot e$ and $ Fp_2 \cdot w = v \cdot e, $ thus 
\[ 
F \pi _1 \cdot Fp_1 \cdot w  = F \pi _1 \cdot u \cdot e = F \pi _1 \cdot v \cdot e =  F \pi _1 \cdot Fp_2 \cdot w
\]
and similarly $F \pi _2 \cdot Fp_1 \cdot w   = F \pi _2 \cdot Fp_2 \cdot w$. In other words, $w$ equalises the pairs of arrows $(F(\pi _1 \cdot p_1), F(\pi _1 \cdot p_2))$ and $(F(\pi _2 \cdot p_1), F(\pi _2 \cdot p_2))$.
If $(E_1 , m_1)$ is the equaliser of 
\[
\pi _1 \cdot p_1, \pi _1 \cdot p_2 \colon (X \times Y) \times (X \times Y) \to X\times Y \to X,
\] 
and $(E_2 , m_2)$ is the equaliser of 
\[
\pi _2 \cdot p_1, \pi _2 \cdot p_2 \colon (X \times Y) \times (X \times Y) \to X\times Y \to Y,
\]
since $F$ preserves equalisers there are unique arrows $w_1\colon T'\to FE_1$ and $w_2\colon T'\to FE_2$ satisfying, respectively, $w = Fm_1 \cdot w_1$ and $w = Fm_2 \cdot w_2$. 

Consider the product diagram
\[\begin{tikzcd}
E_1 & E_1 \times E_2 \arrow{l}[swap]{q_1} \arrow{r}{q_2} & E_2
\end{tikzcd}\]
and the pullback of $\langle w_1, w_2 \rangle$ along $\langle Fq_1 , F q_2 \rangle$, as displayed below.
\[\begin{tikzcd}[column sep = 5em]
T'' \arrow{r}{w''} \arrow[twoheadrightarrow]{d}[swap]{e'} \arrow[dr, phantom, "\lrcorner", very near start, yshift = -0.5em] & F(E_1 \times E_2) \arrow[twoheadrightarrow]{d}{\langle Fq_1 , F q_2 \rangle} \\
T' \arrow{r}{\langle w_1, w_2 \rangle} & FE_1 \times FE_2
\end{tikzcd}\]
Since $\langle Fq_1 , F q_2 \rangle$ is a regular epimorphism by assumption, so is $e'$.
Now, let $(E,m)$ be the equaliser of 
\[
m_1 \cdot q_1 , m_2 \cdot q_2 \colon E_1 \times E_2 \to (X\times Y) \times (X \times Y).
\]
We claim that $w''$ equalises the pair $(F(m_1 \cdot q_1), F(m_2 \cdot q_2))$. 
Just observe that
\begin{equation}\label{eq:second-eq}
Fm_1 \cdot Fq_1 \cdot w'' = Fm_1 \cdot w_1 \cdot e' = w\cdot e'
\end{equation}
and, in a similar fashion,
\begin{equation}\label{eq:third-eq}
Fm_2 \cdot Fq_2 \cdot w'' = Fm_2 \cdot w_2 \cdot e' = w\cdot e'.
\end{equation} 
As $F$ preserves equalisers, there is a unique morphism $z \colon T'' \to FE$ such that $w'' = Fm \cdot z$. Therefore, we have
\begin{align*}
u \cdot e \cdot e' &= Fp_1 \cdot w \cdot e' \\
&= Fp_1 \cdot Fm_1 \cdot Fq_1 \cdot w'' \tag*{eq.~\eqref{eq:second-eq}}\\
&= Fp_1 \cdot Fm_1 \cdot Fq_1 \cdot Fm \cdot z.
\end{align*}
A similar argument, using eq.~\eqref{eq:third-eq}, shows that
\[
v\cdot e \cdot e' = Fp_2 \cdot Fm_2 \cdot Fq_2 \cdot Fm \cdot z.
\]
Since the composite $e\cdot e'$ is a (regular) epimorphism, in order to conclude that $u = v$ it is enough to show that $p_1 \cdot m_1 \cdot q_1 \cdot m=  p_2 \cdot m_2 \cdot q_2 \cdot m$. 
In turn, this follows at once by observing that
\[
\pi _1 \cdot p_1 \cdot m_1 \cdot q_1 \cdot m= \pi _1  \cdot p_2 \cdot m_1 \cdot q_1 \cdot m = \pi _1 \cdot p_2 \cdot m_2 \cdot q_2 \cdot m
\]
and
\[
\pi _2 \cdot p_1 \cdot m_1 \cdot q_1 \cdot m = \pi _2 \cdot p_1 \cdot m_2 \cdot q_2 \cdot m = \pi _2 \cdot p_2 \cdot m_2 \cdot q_2 \cdot m. \qedhere
\]
\end{proof} 

\begin{remark}
If the category $\cat{D}$ in Lemma~\ref{flat-functor-lemma} is a pretopos, then the epimorphicity of the canonical morphisms $\langle F \pi _1 , F \pi _2\rangle \colon F(X\times Y) \to FX \times FY$, along with the preservation of equalisers, implies that $F$ is a \emph{flat} functor into a pretopos. In this situation, the fact that $F$ preserves binary products can be deduced from general considerations about such functors, cf.~\cite{kp}. The argument above is an elementary adaptation of the general one. 
\end{remark} 

Thus, by Lemma~\ref{flat-functor-lemma}, in order to prove that $\F$ preserves binary products, we can reduce ourselves to showing that the canonical maps $\F(X \times Y) \to  \F(X) \times \F(Y)$ are localic surjections. In the next lemma, we shall see that compatible filtrality amounts to exactly this condition in the case where $X$ and $Y$ are filtral objects.
\begin{lemma}\label{l:comp-filtral-rephrasing}
The following statements are equivalent for any filtral category $\K$:
\begin{enumerate}[label=\textup{(}\arabic*\textup{)}]
\item $\K$ is compatibly filtral.
\item\label{i:localic-comparison-filt-prod} For any two filtral objects $S_1, S_2\in \K$, the canonical map
\[
p\colon \F(S_1 \times S_2) \to  \F(S_1) \times \F(S_2)
\]
is a localic surjection.
\end{enumerate}
\end{lemma}

\begin{proof}
Note that, for any distributive lattice $L$ and Boolean algebras $B_1, B_2$, 
\[
\B(L^\op) \cong \B(L)^\op \ \ \text{and} \ \ B_1^\op + B_2^\op \cong (B_1 + B_2)^\op.
\]
Thus, the Boolean algebra homomorphism in eq.~\eqref{eq:BA-hom-inj-filt} is injective precisely when the following one is:
\[
\mu\colon \B(\Sub{S_1}^\op) + \B(\Sub{S_2}^\op) \to \B(\Sub{S_1 \times S_2}^\op).
\]

The Boolean center construction yields a right adjoint to the inclusion $\BA\into \DLat$, and so the restricted functor of ideals $\Idl\colon \BA \to \Frm$ is left adjoint because it is the composition of two left adjoint functors (cf.\ Section~\ref{s:frames-of-ideals}): 
\[\begin{tikzcd}
\BA \arrow[hookrightarrow]{r} & \DLat \arrow{r}{\Idl} & \Frm.
\end{tikzcd}\]
In particular, $\Idl\colon \BA \to \Frm$ preserves colimits. Further, direct inspection shows that it preserves and reflects monomorphisms. Hence, $\mu$ is injective if, and only if, its image
\[
\Idl(\mu)\colon \Idl(\B(\Sub{S_1}^\op)) + \Idl(\B(\Sub{S_2}^\op)) \to \Idl(\B(\Sub{S_1 \times S_2}^\op))
\]
under the functor $\Idl\colon \BA \to \Frm$ is an injective frame homomorphism.

Since $S_1$ and $S_2$ are filtral objects, the domain of $\Idl(\mu)$ can be identified with the frame $\Sub{S_1}^\op + \Sub{S_2}^\op$. We thus have a commutative diagram as follows,
\[\begin{tikzcd}
\Sub{S_1}^\op + \Sub{S_2}^\op \arrow{dr}[swap]{p^*} \arrow{r}{\Idl(\mu)} & \Idl(\B(\Sub{S_1 \times S_2}^\op)) \arrow{d}{\lambda^*} \\
{} & \Sub{S_1 \times S_2}^\op
\end{tikzcd}\]
where $p^*$ is the frame homomorphism corresponding to the localic map $p$ in item~\ref{i:localic-comparison-filt-prod}, and $\lambda^*$ is the unique frame homomorphism (induced by the universal property of the left adjoint $\Idl\colon \DLat\to \Frm$) extending the inclusion of distributive lattices 
\[
\B(\Sub{S_1 \times S_2}^\op)\into \Sub{S_1 \times S_2}^\op.
\]
Note that the codomain of the localic map $p$ is given by a product in $\CHLoc$, whereas the coproduct $\Sub{S_1}^\op + \Sub{S_2}^\op$ is computed in the category $\Frm$; however, the two are dual to each other because $\CHLoc$ is closed under limits in $\Loc$ (see Section~\ref{s:frames-of-ideals}).
Now, the right adjoint $\lambda_*$ of $\lambda^*$ sends $S\in \Sub{S_1 \times S_2}^\op$ to the ideal 
\[
\lambda_*(S)\coloneqq \{U\in \B(\Sub{S_1 \times S_2}^\op)\mid U \sleq S\}.
\]
It follows that $\lambda_*(0) = 0$ and so, by Lemma~\ref{l:dense-implies-surjective}, $\lambda^*$ is injective. 
The statement of the lemma follows by observing that, since $\lambda^*\cdot \Idl(\mu) = p^*$ and $\lambda^*$ is injective, $\Idl(\mu)$ is injective if, and only if, $p^*$ is injective, i.e.\ $p$ is a localic surjection.
\end{proof}

\begin{proposition}\label{product-localic-surj}
Let $\K$ be a compatibly filtral category with enough subobjects. The functor $\F\colon \K\to\CHLoc$ preserves binary products.
\end{proposition}
\begin{proof}
In view of Proposition~\ref{p:preservation-of-equalizers} and Lemma~\ref{flat-functor-lemma}, it suffices to show that for any two objects $X$ and $Y$ of $\K$, the canonical map
\[
\F(X \times Y) \to  \F(X) \times \F(Y)
\]
is a localic surjection.

Consider regular epimorphisms $\epsilon_1\colon S_1 \repi X$ and $\epsilon_2\colon S_2 \repi Y$ such that $S_1$ and $S_2$ are filtral objects. Since $\F$ preserves regular epimorphisms by Proposition~\ref{p:functor-to-CHLoc}, and the product of regular epimorphisms in a regular category is again a regular epimorphism (see e.g.\ \cite[Proposition~1.12, p.~134]{BarrGrilletOsdol1971}), we obtain a commutative square as follows,
\[\begin{tikzcd}
\F(S_1 \times S_2) \arrow{r} \arrow[twoheadrightarrow]{d}[swap]{\F(\epsilon_1\times\epsilon_2)} &  \F(S_1) \times \F(S_2) \arrow[twoheadrightarrow]{d}{\F(\epsilon_1)\times\F(\epsilon_2)} \\
\F(X \times Y) \arrow{r} & \F(X) \times \F(Y)
\end{tikzcd}\]
where the horizontal arrows are the canonical ones. The top horizontal arrow is a localic surjection by Lemma~\ref{l:comp-filtral-rephrasing}, hence the bottom horizontal one is also a localic surjection.
\end{proof}

\subsection{Terminal object: non-triviality}\label{ss:non-triviality}
Concerning the preservation of the terminal object, we note the following easy fact:
\begin{lemma}\label{l:terminal-to-subterminal}
Let $F\colon \cat{C} \to \cat{D}$ be a functor. Suppose $\cat{C}$ and $\cat{D}$ admit a terminal object and $F$ preserves binary products. Then the unique arrow $F(\one_{\cat{C}}) \to \one_{\cat{D}}$ is a monomorphism. 
\end{lemma}

\begin{proof}
Note that the unique morphism $F(\one_{\cat{C}}) \to \one_{\cat{D}}$ is monic if, and only if, any two morphisms into $F(\one_{\cat{C}})$ are equal. In $\cat{C}$, the pair of identity arrows $\one_{\cat{C}} \leftarrow \one_{\cat{C}} \rightarrow \one_{\cat{C}}$ is a product diagram, hence the pair of identity arrows
\[
F(\one_{\cat{C}}) \leftarrow F(\one_{\cat{C}}) \rightarrow F(\one_{\cat{C}})
\]
is a product diagram in $\cat{D}$. It follows that, for any two arrows $u,v \colon T \to F(\one_{\cat{C}})$, their pairing $\langle u, v \rangle \colon T \to F(\one_{\cat{C}})$ satisfies $u = \langle u, v \rangle =v$.
\end{proof}

Let $\Omega$ denote the frame of truth values, which is also the initial object of $\Frm$ (see e.g.\ \cite[Lemma~1.5.1]{t}). Because $\CHLoc$ is closed in $\Loc$ under finite limits, $\Omega$ is also the frame corresponding to the terminal object of $\CHLoc$. To avoid confusion, we shall write $\mho$ for the terminal object of $\CHLoc$ (equivalently, of $\Loc$); that is, $O\mho = \Omega$. 
When applied to the functor $\F\colon\K\to\CHLoc$, for $\K$ a compatibly filtral category with enough subobjects, Lemma~\ref{l:terminal-to-subterminal} shows that the unique map 
\[
t \colon \F(\one) \to \mho
\] 
is a localic injection. 
From the intuitionistic point of view, the sufficiently strong way to say that a filtral category is \emph{non-trivial} is to say that $t$ is a surjection, i.e.\ $t^*\colon \Omega \to \Sub{\one}^\op$ is injective, cf.\ \cite[\S 2.21]{fs}. That is, according to the usual terminology in constructive locale theory, a filtral category is non-trivial if $\Sub{\one}^\op$ is a \emph{positive} locale.

\begin{example}
Clearly, $\CHLoc$ is non-trivial because the unique frame homomorphism $\Omega \to \Sub{\mho}^\op$ is an isomorphism.
\end{example}

An application of Lemma~\ref{l:terminal-to-subterminal}, combined with Propositions~\ref{p:preservation-of-equalizers} and~\ref{product-localic-surj}, yields the following result.
\begin{proposition}\label{p:preservation-fin-lim}
Let $\K$ be a non-trivial compatibly filtral category with enough subobjects. The functor $\F\colon\K\to\CHLoc$ is Cartesian.
\end{proposition}

\section{A pretopos embedding into compact Hausdorff locales}\label{sec:main}
In the previous section we saw that, if $\K$ is a non-trivial compatibly filtral category with enough subobjects, $\F\colon\K\to\CHLoc$ is Cartesian.
In this section we shall prove that, if we further assume that $\K$ is a pretopos, the functor $\F$ is a pretopos embedding. In other words, $\K$ is equivalent to a sub-pretopos of $\CHLoc$. The essential image of this embedding is studied in Section~\ref{s:essential image}.

The key ingredient to establish that $\F$ is a pretopos morphism (see Theorem~\ref{t:Sub-pretopos-morphism} below) consists in showing that finite coproducts are preserved:
\begin{proposition}\label{p:preserv-coproducts}
For a filtral pretopos $\K$, the functor $\F \colon \K \to \CHLoc$ preserves finite coproducts.
\end{proposition}
\begin{proof}
The initial object in a coherent category is strict, hence $\Sub{\zero}^\op$ is the one-element frame and the corresponding locale $\F(\zero)$ is initial in $\CHLoc$.
Now, consider a coproduct diagram
\begin{equation}\label{eq:coproduct}
\begin{tikzcd}
X \arrow{r}{i_1} & X+Y & Y \arrow{l}[swap]{i_2}
\end{tikzcd}
\end{equation} 
in $\K$, and its image in $\CHLoc$ under the functor $\F$:
\[\begin{tikzcd}
\F(X) \arrow{r}{\F(i_1)} & \F(X+Y) & \F(Y). \arrow{l}[swap]{\F(i_2)}
\end{tikzcd}\]
Products in the category of frames are computed in the category of sets, and direct inspection shows that compact Hausdorff (equivalently, compact regular) locales are closed under binary coproducts in the category of locales. 
Thus, the functor $\F$ preserves the coproduct diagram in eq.~\eqref{eq:coproduct} if, and only if, the unique frame homomorphism $\zeta$ making the following diagram commute is an isomorphism,
\[\begin{tikzcd}[row sep = 2em]
{} & \Sub{X+Y}^\op \arrow{dl}[swap]{i_1^{-1}} \arrow{dr}{i_2^{-1}} \arrow{d}[description]{\zeta} & {} \\
\Sub{X}^\op & \Sub{X}^\op\times \Sub{Y}^\op \arrow{l} \arrow{r} & \Sub{Y}^\op
\end{tikzcd}\]
where the bottom row is a product diagram in $\Frm$.

Observe that, since $\K$ is extensive, for all subobjects $m\colon S\mono X+Y$ the top row in the following diagram is a coproduct.
\begin{equation}\label{eq:extensivity}
\begin{tikzcd}
i_1^{-1}(S) \arrow[rightarrowtail]{d}[swap]{\lambda_m} \arrow{r} \arrow[dr, phantom, "\lrcorner", very near start] & S \arrow[rightarrowtail]{d}[description]{m} & i_2^{-1}(S) \arrow[rightarrowtail]{d}{\rho_m} \arrow{l} \arrow[dl, phantom, "\llcorner", very near start] \\
X \arrow{r}{i_1} & X+Y & Y \arrow{l}[swap]{i_2}
\end{tikzcd}
\end{equation}
Hence, $\zeta$ can be described explicitly as follows:
\[
\zeta\colon \Sub{X+Y}^\op\to \Sub{X}^\op\times \Sub{Y}^\op, \ \ \zeta(m)=(\lambda_m,\rho_m).
\]
To see that $\zeta$ is injective, suppose that $m\colon S\mono X+Y$ and $n\colon T\mono X+Y$ are subobjects such that $\zeta(m)=\zeta(n)$. Thus, in the slice category $\K/{(X+Y)}$, 
\[
i_1 \cdot \lambda_m \cong i_1\cdot \lambda_n \ \text{ and } \ i_2 \cdot \rho_m \cong i_2\cdot \rho_n.
\]
On the other hand, by eq.~\eqref{eq:extensivity} we have
\[
m \cong i_1\cdot \lambda_m + i_2 \cdot \rho_m \ \text{ and } \ n \cong i_1\cdot \lambda_n + i_2 \cdot \rho_n
\]
in $\K/{(X+Y)}$, and so $m\cong n$ as objects of the slice category. Equivalently, $m=n$ as subobjects of $X+Y$.

It remains to show that $\zeta$ is surjective, hence a frame isomorphism.
Consider a pair $(m_1,m_2)\in \Sub{X}^\op\times \Sub{Y}^\op$ given by subobjects $m_1\colon S_1\mono X$ and $m_2\colon S_2\mono Y$.
Let 
\[
m_1+m_2\colon S_1+S_2\to X+Y
\] 
be the coproduct of $i_1\cdot m_1$ and $i_2\cdot m_2$ in $\K/{(X+Y)}$, and let ${\iota_1\colon S_1\to S_1+S_2}$ and ${\iota_2\colon S_2\to S_1+S_2}$ be the coproduct arrows. Since $\K$ is extensive, the following diagram comprises two pullback squares.
\[\begin{tikzcd}[row sep=3em, column sep=3em]
S_1 \arrow{r}{\iota_1} \arrow[rightarrowtail]{d}[swap]{m_1} \arrow[dr, phantom, "\lrcorner", very near start] & S_1+S_2 \arrow{d}[description]{m_1+m_2} & S_2 \arrow{l}[swap]{\iota_2} \arrow[rightarrowtail]{d}{m_2} \arrow[dl, phantom, "\llcorner", very near start] \\
X \arrow{r}{i_1} & X+Y & Y \arrow{l}[swap]{i_2}
\end{tikzcd}\]
Consider the (regular epi, mono) factorisation of $m_1+m_2$, as displayed below:
\[\begin{tikzcd}
S_1+S_2 \arrow[twoheadrightarrow]{r}{e} & T \arrow[rightarrowtail]{r}{n} & X+Y.
\end{tikzcd}\] 
We claim that $\zeta(n)=(m_1,m_2)$. 
Because image factorisations in a regular category are pullback-stable, the pullback of $n$ along $i_1$, i.e.\ $\lambda_n$, coincides with the image of the pullback of $m_1+m_2$ along $i_1$. By the diagram above, the latter pullback coincides with $m_1\colon S_1\mono X$. Therefore, $\lambda_n = m_1$ as subobjects of $X$. Similarly, we get $\rho_n= m_2$ as subobjects of $Y$, showing that $\zeta(n)=(m_1,m_2)$.
\end{proof}

We thus obtain our first main result, which identifies sufficient conditions ensuring that a pretopos can be embedded into the pretopos $\CHLoc$ of compact Hausdorff locales. 
\begin{theorem}\label{t:Sub-pretopos-morphism}
Let $\K$ be a non-trivial compatibly filtral pretopos with enough subobjects. 
The functor $\F \colon \K \to \CHLoc$ is a pretopos embedding.
\end{theorem}

\begin{proof}
The functor $\F \colon \K \to \CHLoc$ is Cartesian and preserves regular epimorphisms by Propositions~\ref{p:preservation-fin-lim} and~\ref{p:functor-to-CHLoc}, respectively. Moreover, it preserves finite coproducts by Proposition~\ref{p:preserv-coproducts}, and so Lemma~\ref{l:pretopos-morphism-iff-coherent} entails that $\F$ is a pretopos morphism.

Finally, $\F$ is bijective on subobjects by Lemma~\ref{l:full-on-subobjects}, and thus a pretopos embedding by Lemma~\ref{l:pretopos-mor-full-faithful}.
\end{proof}

From Theorem~\ref{t:Sub-pretopos-morphism} we can derive the following characterisation, up to weak equivalence, of the category of compact Hausdorff locales. This characterisation is constructive, hence valid in the logic of any topos, but somewhat unsatisfactory because it involves a universal quantification over the class of all Boolean algebras. In Section~\ref{s:essential image} we will refine Theorem~\ref{t:Sub-pretopos-morphism} in a different direction, by studying the essential image of $\F$ under the assumption that the law of weak excluded middle holds.

\begin{corollary}\label{cor:essent-surj}
Let $\K$ be a non-trivial compatibly filtral pretopos with enough subobjects. 
Then $\K$ is weakly equivalent to $\CHLoc$ if, and only if, the following holds:
\begin{enumerate}[label=\textnormal{(\textbf{B})}]
\item\label{condition-B} For every Boolean algebra $B$ there exist an object $X\in\K$ and an injective lattice homomorphism $B\into \Sub{X}$.
\end{enumerate}
\end{corollary}

\begin{proof}
Fix a non-trivial compatibly filtral pretopos $\K$ with enough subobjects. Clearly, if there exists a weak equivalence $\K\to\CHLoc$, then $\K$ satisfies~\ref{condition-B}.

Conversely, recall from Section~\ref{s:pretoposes} that a pretopos morphism is a weak equivalence precisely when it is an embedding that covers its codomain. By Theorem~\ref{t:Sub-pretopos-morphism}, combined with the fact that every compact Hausdorff locale is covered by a Stone locale, the functor $\F \colon \K \to \CHLoc$ is a weak equivalence if, and only if, the following condition holds:
\begin{enumerate}[label=\textnormal{(\textbf{B}')}]
\item\label{condition-B'} For every Boolean algebra $B$ there exist an object $X\in\K$ and an injective frame homomorphism $\Idl(B)\into \Sub{X}^{\op}$.
\end{enumerate}

We show that~\ref{condition-B} $\Rightarrow$ \ref{condition-B'} (in fact, the two conditions are equivalent). For any Boolean algebra $B$, condition~\ref{condition-B} applied to $B^{\op}$ yields the existence of an object $Y\in \K$ and an injective lattice homomorphism $B\into \Sub{Y}^{\op}$. Let $X$ be a filtral object of $\K$ covering~$Y$. Since $\F\colon \K\to\CHLoc$ preserves (regular) epimorphisms by Proposition~\ref{p:functor-to-CHLoc}, we get an injective frame homomorphism $\Sub{Y}^{\op}\into \Sub{X}^{\op}$; the composite 
\[
B\into \Sub{Y}^{\op}\into \Sub{X}^{\op}
\]
then corestricts to an injective Boolean algebra homomorphism $B\into \B(\Sub{X}^{\op})$. The functor $\Idl\colon \BA \to \Frm$ preserves monomorphisms, thus it sends the latter morphism to an injective frame homomorphism
\[
\Idl(B)\into \Idl(\B(\Sub{X}^{\op}))\cong \Sub{X}^{\op}. \qedhere
\]
\end{proof}

\section{The essential image of the embedding}\label{s:essential image}
All the results we have obtained so far---in particular, the construction of the pretopos embedding $\F \colon \K \to \CHLoc$ in Theorem~\ref{t:Sub-pretopos-morphism}---are valid in the internal logic of any topos.
In order to study the essential image of $\F$, in this section we shall consider the principle of \emph{weak excluded middle} stating that, for any proposition $P$,
\[
\neg P \vee \neg \neg P
\]
holds. The latter principle is equivalent to \emph{de Morgan's law},\footnote{This is the only one of the four laws, collectively named de Morgan's laws and expressing the fact that negation is a lattice anti-isomorphism in a Boolean algebra, that is not an intuitionistic tautology.} stating that 
\[
\neg (P \wedge Q) \to (\neg P \vee \neg Q)
\]
holds for all propositions $P$ and $Q$. A \emph{de Morgan topos} is one whose internal logic satisfies de Morgan's law. Henceforth, we mark with $\WEM$ the results that rely on the principle of weak excluded middle.

\begin{remark}
In a de Morgan topos, compact Hausdorff locales need not be spatial. In fact, every non-Boolean topos contains an internal compact Hausdorff locale (and even a Stone locale) that is not spatial; cf.\ Lemma~\ref{l:enough-points-lem} below.
\end{remark}

\begin{example}
The topos of sheaves on a space $X$ is de Morgan if, and only if, $X$ is extremally disconnected. More generally, the topos of sheaves on a locale $X$ is de Morgan precisely when $X$ is extremally disconnected (i.e.\ $OX$ satisfies de Morgan's law), and it is not Boolean unless $OX$ is a Boolean algebra.  Moreover, the topos of presheaves on a small category $\cat{C}$ is de Morgan precisely when $\cat{C}$ satisfies the \emph{right Ore condition}, which states that any pair of arrows in $\cat{C}$ with common codomain can be completed to a commutative square. 
For example, if $M$ is a commutative monoid then the topos $\Set^{M}$ of $M$-sets is de Morgan, and it is not Boolean unless $M$ is a group.
Cf.\ e.g.\ \cite[\S D4.6]{Elephant2}.
\end{example}

Let $\K$ be a pretopos satisfying the assumptions of Theorem~\ref{t:Sub-pretopos-morphism}.
Our second main result is valid in the internal logic of a de Morgan topos and shows that if $\K$ admits set-indexed copowers of its terminal object, the essential image of $\F$ contains all spatial compact Hausdorff locales and their closed sublocales. 

\begin{theorem}[$\WEM$]\label{th:spatial-locales-covered}
Let $\K$ be a non-trivial compatibly filtral pretopos with enough subobjects, and assume that the terminal object of $\K$ admits all set-indexed copowers.
Then the essential image of the pretopos embedding $\F \colon \K \to \CHLoc$ contains all spatial compact Hausdorff locales and their closed sublocales.
\end{theorem}

In the remainder of this section, we offer a proof of Theorem~\ref{th:spatial-locales-covered}. Fix a pretopos~$\K$ satisfying the assumptions of the aforementioned result and recall from Theorem~\ref{t:Sub-pretopos-morphism} that the functor $\F$ is a pretopos embedding, and is bijective on subobjects by Lemma~\ref{l:full-on-subobjects}. Since subobjects in $\CHLoc$ can be identified with closed sublocales, if the essential image of $\F$ contains all spatial compact Hausdorff locales then it also contains their closed sublocales. Thus, in view of Lemma~\ref{l:cover-essential-image}, it suffices to show that $\F$ covers all spatial compact Hausdorff locales.
Fix an arbitrary spatial compact Hausdorff locale $X$. If $\pts X$ denotes the set of points of $X$, the canonical map
\[
\pi\colon\coprod_{\pts X}{\mho} \to X
\]
from the ($\pts X$)-fold copower in $\Loc$ of the terminal locale $\mho$ is a localic surjection. Hence its unique extension $\varpi$ to the Stone-\v{C}ech compactification of $\coprod_{\pts X}{\mho}$, induced by the universal property of the left adjoint $\beta\colon \Loc\to\CHLoc$, is also a localic surjection.
\[\begin{tikzcd}
\displaystyle\coprod_{\pts X}{\mho} \arrow[twoheadrightarrow]{dr}[swap]{\pi} \arrow{r} & \beta\left(\displaystyle\coprod_{\pts X}{\mho}\right) \arrow[twoheadrightarrow]{d}{\varpi} \\
{} & X
\end{tikzcd}\]
Consider the copower $\coprod_{\pts X}{\one}$ in $\K$. We will prove that there exists a localic surjection
\begin{equation}\label{eq:main-localic-surj}
\F\left(\coprod_{\pts X}{\one}\right) \repi \beta\left(\coprod_{\pts X}{\mho}\right).
\end{equation}
Composing the latter with $\varpi$, we will obtain a localic surjection $\F(\coprod_{\pts X}{\one}) \repi X$, thus showing that the functor $\F$ covers $X$.

The frame corresponding to the locale $\coprod_{\pts X}{\mho}$ is the power-set $\P(\pts X)$. Recall from Section~\ref{s:frames-of-ideals} that 
\[
O\left(\beta\left(\coprod_{\pts X}{\mho}\right)\right)
\] 
can be identified with a subframe of $\Idl(\P(\pts X))$.\footnote{Note that $\P(\pts X)$ need \emph{not} be a Boolean algebra, as we are not assuming the principle of excluded middle. Hence, $O(\beta(\coprod_{\pts X}{\mho}))$ may be a proper subframe of $\Idl(\P(\pts X))$.}
Thus, in order to establish the existence of a localic surjection as in eq.~\eqref{eq:main-localic-surj}, it is enough to exhibit an injective frame homomorphism
\[
\Idl(\P(\pts X)) \mono \Sub{\coprod_{\pts X}{\one}}^\op.
\]
More generally, for any set $S$ we will construct an injective frame homomorphism 
\begin{equation*}
\Idl(\P(S)) \mono \Sub{\coprod_{S}{\one}}^\op.
\end{equation*}

The strategy is as follows: we shall define a lattice homomorphism
\[
\alpha\colon \P(S) \to \Sub{\coprod_{S}{\one}}^\op
\]
and show that its unique extension to a frame homomorphism
\[
\alpha^*\colon \Idl(\P(S)) \to \Sub{\coprod_{S}{\one}}^\op
\]
is injective.
In order to define the map $\alpha$, we use the following lemma.
Recall that, in any extensive category, the coproduct maps are monomorphisms (see e.g.\ \cite[Proposition~2.6]{CLW93}).
In particular, for each $x\in S$, the coproduct map 
\[
i_x\colon \one \to \coprod_{S}{\one}
\]
is a monomorphism and we shall identify it with a subobject of $\coprod_{S}{\one}$.

\begin{lemma}\label{l:complemented-subobj}
The following statements hold for any complemented ${P\in \P(S)}$:\footnote{A complemented element of $\P(S)$ is usually called a \emph{decidable subset} of $S$.}
\begin{enumerate}[label=\textup{(}\alph*\textup{)}]
\item\label{i:iP-supremum} The obvious morphism 
\[
i_P\colon \coprod_{P}{\one}\to \coprod_{S}{\one}
\] 
 is a monomorphism and $i_P = \displaystyle\bigvee{\{i_x\mid x\in P\}}$ in $\Sub{\coprod_{S}{\one}}$. In particular, the top element of $\Sub{\coprod_{S}{\one}}$ coincides with the supremum of the set $\{i_x\mid x\in S\}$.
 \item\label{i:iP-complemented} $i_P$ is a complemented element of $\Sub{\coprod_{S}{\one}}$ with complement $i_{P^c}$.
 \item\label{i:complemented-extensionality} $P = \{x\in S \mid i_x\leq i_P \text{ in } \Sub{\coprod_{S}{\one}}\}$.
 \end{enumerate}
\end{lemma}

\begin{proof}
\ref{i:iP-supremum} Note that $P\cap P^c=\emptyset$ and, since $P$ is complemented, also $P\cup P^c =S$. Thus, the following is a coproduct diagram in $\K$.
\[\begin{tikzcd}
\displaystyle\coprod_{P}{\one} \arrow{r}{i_P} & \displaystyle\coprod_{S}{\one} & \displaystyle\coprod_{P^c}{\one} \arrow{l}[swap]{i_{P^c}}
\end{tikzcd}\] 
As coproduct maps in $\K$ are monomorphisms, $i_P$ is a monomorphism (and so is $i_{P^c}$). Since the supremum in $\Sub{\coprod_{S}{\one}}$ of the set $\{i_x\mid x\in P\}$ is computed by taking the (regular epi, mono) factorisation of $i_P$, we see that $i_P = \displaystyle\bigvee{\{i_x\mid x\in P\}}$. 

\ref{i:iP-complemented} Because binary coproducts in an extensive category are disjoint, $i_P \wedge i_{P^c}=0$. Moreover, $i_P \vee i_{P^c}=1$ because, in a coherent category, the supremum of two disjoint subobjects coincides with their coproduct (see e.g.\ \cite[Corollary~A.1.4.4]{el1}).

\ref{i:complemented-extensionality} We need to prove that, for all $x\in S$,
\[
x\in P \ \Longleftrightarrow \ i_x \leq i_P.
\]
As $i_P = \displaystyle\bigvee{\{i_x\mid x\in P\}}$ by item~\ref{i:iP-supremum}, the left-to-right implication is immediate. For the other direction, suppose $i_x \leq i_P$. Since $P$ is complemented, it suffices to show that $x\notin P^c$. If $x\in P^c$ then $i_x\leq i_{P^c}$ and so, in $\Sub{\coprod_{S}{\one}}$, we have
\[
0< i_x \leq i_P \wedge i_{P^c} = 0. 
\]
Therefore, $x\notin P^c$.
\end{proof}

By virtue of the principle of weak excluded middle, for each $P\in \P(S)$, the set $P^c$ is a complemented element of $\P(S)$ with complement $P^{cc}$. Therefore, Lemma~\ref{l:complemented-subobj}\ref{i:iP-supremum} tells us that the following map is well defined:
\[
\alpha\colon \P(S) \to \Sub{\coprod_{S}{\one}}^\op, \ \ \alpha(P) \coloneqq i_{P^c}.
\]

\begin{lemma}[$\WEM$]
The map $\alpha$ is a lattice homomorphism.
\end{lemma}

\begin{proof}
The bounds $\emptyset, S\in \P(S)$ are complements of each other and easily seen to be preserved.
If $P,Q\in \P(S)$, then 
\begin{align*}
\alpha(P \cap Q) &= i_{(P\cap Q)^c} \\
& = i_{P^c \cup Q^c} \tag*{de Morgan's law} \\
&= \bigvee{\{i_x \mid x\in P^c \cup Q^c\}} \tag*{Lemma~\ref{l:complemented-subobj}\ref{i:iP-supremum}} \\
&= \bigvee{\{i_x \mid x\in P^c\}} \vee  \bigvee{\{i_x \mid x\in Q^c\}} \\
& = \alpha(P) \swedge \alpha(Q). \tag*{Lemma~\ref{l:complemented-subobj}\ref{i:iP-supremum}}
\end{align*}
Moreover, an application of Lemma~\ref{l:complemented-subobj}\ref{i:iP-complemented} yields, for all $P\in \P(S)$,
\[
\alpha(P^c) = i_{P^{cc}} = \neg i_{P^c} = \neg \alpha(P),
\]
showing that $\alpha$ preserves pseudocomplements. To conclude, it suffices to observe that, for any map $h\colon A\to B$ from a Heyting algebra $A$ to a Boolean algebra $B$, if $h$ preserves binary infima and pseudocomplements then it also preserves binary suprema. 
Since the image of $\alpha$ is contained in the Boolean center of $\Sub{\coprod_{S}{\one}}^\op$ by Lemma~\ref{l:complemented-subobj}\ref{i:iP-complemented}, it follows that $\alpha$ preserves binary suprema and thus is a lattice homomorphism.
\end{proof}

Denote by 
\[
\alpha^*\colon \Idl(\P(S)) \to \Sub{\coprod_{S}{\one}}^\op
\]
the unique frame homomorphism, induced by the universal property of the left adjoint $\Idl\colon \DLat\to \Frm$, extending the lattice homomorphism $\alpha$. For all ideals $J\in \Idl(\P(S))$,
\[
\alpha^*(J) = \bigsvee{\{\alpha(P)\mid P\in J\}}. 
\]
The right adjoint $\alpha_*$ of $\alpha^*$ is given by
\[
\alpha_*\colon \Sub{\coprod_{S}{\one}}^\op\to \Idl(\P(S)), \ \ \alpha_*(T) = \{P\in\P(S) \mid i_{P^c} \sleq T\}
\]
and satisfies
\begin{align*}
\alpha_*(0) &= \{P\in\P(S) \mid  i_{P^c} \sleq 0\} \\
&= \{P\in\P(S) \mid  \forall x\in S, \ i_x \leq i_{P^c}\} \tag*{Lemma~\ref{l:complemented-subobj}\ref{i:iP-supremum}} \\
&= \{P\in\P(S) \mid  P^c = S\} \tag*{Lemma~\ref{l:complemented-subobj}\ref{i:complemented-extensionality}} \\
&= \{\emptyset\},
\end{align*}
which is the least element of $\Idl(\P(S))$. Therefore, $\alpha^*\colon \Idl(\P(S)) \to \Sub{\coprod_{S}{\one}}^\op$ is an injective frame homomorphism by Lemma~\ref{l:dense-implies-surjective}. This concludes the proof of Theorem~\ref{th:spatial-locales-covered}.

\section{The case of compact Hausdorff spaces}\label{s:CH-spaces}
In this section we show how to derive Marra and Reggio's characterisation of the category of compact Hausdorff spaces, stated in Theorem~\ref{th:mr-main}, from Theorem~\ref{th:spatial-locales-covered}. Note that, under the assumptions of the latter theorem, the pretopos embedding ${\F \colon \K \to \CHLoc}$ is a weak equivalence provided that compact Hausdorff locales have enough points. This inevitably leads us to classical logic; indeed, we have the following known fact:

\begin{lemma}\label{l:enough-points-lem}
The following statements are equivalent in the logic of any topos:
\begin{enumerate}[label=\textup{(}\arabic*\textup{)}]
\item\label{i:CH-enough-points} Compact Hausdorff locales have enough points.
\item\label{i:Stone-enough-points} Stone locales have enough points.
\item\label{i:Stone-repr-thm} The Stone representation theorem for Boolean algebras\footnote{This asserts that every Boolean algebra is isomorphic to a subalgebra of $\P(S)$ for some set $S$.} holds.
\end{enumerate}
Moreover, each of the previous equivalent properties implies the law of excluded middle.
\end{lemma}

\begin{proof}
Every compact Hausdorff locale is covered by a Stone locale, thus \ref{i:CH-enough-points} $\Leftrightarrow$ \ref{i:Stone-enough-points}.

In \cite[\S 1.8]{t}, Townsend showed that the existence of enough points for Stone locales is equivalent, in the logic of a topos, to the following statement: For every Boolean algebra $B$ and every $a\in B$, if $f(a)=0$ for all lattice homomorphisms $f \colon B \to \Omega$, then $a=0$. 
That is, the canonical lattice homomorphism 
\[
B\to \prod_{\DLat(B,\Omega)}{\Omega}
\] 
is injective.
In turn, the latter statement is equivalent to the Stone representation theorem for Boolean algebras, and so \ref{i:Stone-enough-points} $\Leftrightarrow$ \ref{i:Stone-repr-thm}.

Finally, Bell showed in \cite[Theorem~3.2]{Bell1999} that the Stone representation theorem for Boolean algebras constructively implies the law of excluded middle.
\end{proof}

\begin{remark}
Classically, the existence of enough points for compact Hausdorff locales can be derived from the prime ideal theorem for distributive lattices (equivalently, for Boolean algebras). A constructive version of the latter principle, the \emph{constructive prime ideal theorem} (CPIT), was introduced by Townsend~\cite{t} and states that, for any distributive lattice $D$ and $a\in D$, if $f(a)=0$ for all lattice homomorphisms $f \colon D \to \Omega$, then $a=0$. Constructively, CPIT is equivalent to the existence of enough points for compact Hausdorff locales; moreover, in the presence of the law of excluded middle, it is equivalent to the prime ideal theorem for distributive lattices. See \cite[\S 1.8]{t}. Thus, by Lemma~\ref{l:enough-points-lem}, every topos satisfying CPIT is Boolean and satisfies the prime ideal theorem.
\end{remark}

Even in the presence of excluded middle, in order to derive Theorem~\ref{th:mr-main} from Theorem~\ref{th:spatial-locales-covered} we need to show that compatible filtrality follows from the remaining assumptions. 
To this end, borrowing terminology from topos theory, let us say that an object of a category is an \emph{atom} if it has precisely two (distinct) subobjects.
We shall see that a filtral category with enough subobjects is compatibly filtral whenever its terminal object $\one$ is an atom (Proposition~\ref{p:one-atom-comp-filtral}). 
We start with a technical lemma:
\begin{lemma}\label{l:products-lemma-rectangles}
Consider a product diagram
\[\begin{tikzcd}
X_1 & X_1 \times X_2 \arrow{l}[swap]{\pi_1} \arrow{r}{\pi_2} & X_2
\end{tikzcd}\]
in a coherent category $\K$ whose terminal object is an atom. The following statements~hold:
\begin{enumerate}[label=\textup{(}\alph*\textup{)}]
\item\label{constant-0} $\forall \, U\in \Sub{X_1}^\op$, either $U= 1$ or $\pi_2 [\pi_1^{-1}(U)] = 0\in \Sub{X_2}^\op$.
\item\label{rectangles-simplified} If $\K$ is filtral, for all complemented elements $T_1 \in \Sub{X_1}^\op$ and $T_2 \in \Sub{X_2}^\op$,
    \[
    \pi_1^{-1}(T_1) \swedge \pi_2^{-1}(T_2) = 0 \ \Longleftrightarrow \ (T_1 = 0 \ \text{ or } \ T_2 = 0).
      \]
    \end{enumerate}
\end{lemma}
\begin{proof}
\ref{constant-0} We prove the order-dual statement. 
Fix a subobject $m\colon U\mono X_1$ and let $U \repi V \mono \one$ be the (regular epi, mono) factorisation of the composite 
\[
!_1\cdot m\colon U\mono X_1 \to \one.
\]
Since $\one$ is an atom, either $V\cong \zero$ or $V\cong \one$. If $V\cong \zero$ we have $U \cong \zero$ because the initial object of a coherent category is strict, and so $U=0\in\Sub{X_1}$. It remains to show that if $V\cong \one$, then $\pi_2 [\pi_1^{-1}(U)] = 1\in \Sub{X_2}$.

Suppose $V\cong \one$. Then $!_1\cdot m$ is a regular epimorphism and so $!_1[U] = 1\in \Sub{\one}$. The map $!_2^{-1}\colon \Sub{\one}\to\Sub{X_2}$ preserves the top element because it is a bounded lattice homomorphism, hence $!_2^{-1} !_1[U] = 1\in \Sub{X_2}$. Now, consider the following pullback square:
\[\begin{tikzcd}
X_1 \times X_2 \arrow{r}{\pi_1} \arrow{d}[swap]{\pi_2} \arrow[dr, phantom, "\lrcorner", very near start] & X_1 \arrow{d}{!_1} \\
X_2 \arrow{r}{!_2} & \one
\end{tikzcd}\]
The Beck--Chevalley condition 
\[
\pi_2 [\pi_1^{-1}(-)] = {!}_2^{-1} !_1[-]\colon \Sub{X_1}\to\Sub{X_2}
\]
holds in any regular category (cf.\ e.g.\ \cite[pp.~101--102]{mrey}), and so $\pi_2 [\pi_1^{-1}(U)] = !_2^{-1} !_1[U] = 1$.

\ref{rectangles-simplified} The right-to-left direction is immediate, since the frame homomorphisms $\pi_1^{-1}$ and $\pi_2^{-1}$ preserve $0$. 
For the left-to-right direction, suppose that 
\[
\pi_1^{-1}(T_1)\swedge \pi_2^{-1}(T_2) = 0
\] 
and denote by $U\in\Sub{X_1}^\op$ the complement of $T_1$. We have
\begin{align}\label{eq:last-step}
\pi_1^{-1}(T_1)\swedge \pi_2^{-1}(T_2) & = 0 \nonumber \\ 
\Longrightarrow \ (\pi_1^{-1}(T_1)\svee \pi_1^{-1}(U)) \swedge (\pi_2^{-1}(T_2) \svee \pi_1^{-1}(U)) & = \pi_1^{-1}(U) \nonumber \\
\Longrightarrow \ \pi_1^{-1}(T_1\svee U) \swedge (\pi_2^{-1}(T_2) \svee \pi_1^{-1}(U)) & = \pi_1^{-1}(U) \nonumber \\
\Longrightarrow \ 1 \swedge (\pi_2^{-1}(T_2) \svee \pi_1^{-1}(U)) & = \pi_1^{-1}(U) \nonumber \\
\Longrightarrow \ \pi_2[\pi_2^{-1}(T_2) \svee \pi_1^{-1}(U)] & = \pi_2[\pi_1^{-1}(U)] \nonumber \\
\Longrightarrow \ T_2 \svee \pi_2[\pi_1^{-1}(U)] & = \pi_2[\pi_1^{-1}(U)]
\end{align}
where the last implication follows from the fact that, since $\K$ is filtral, the localic map 
\[
\pi_2[-] = \F(\pi_2) \colon \F(X_1\times X_2) \to \F(X_2)
\] 
is closed.
By item~\ref{constant-0}, either $U= 1$ or $\pi_2 [\pi_1^{-1}(U)] = 0$. If $U=1$ then $T_{1}=0$, whereas if $\pi_2 [\pi_1^{-1}(U)] = 0$, then eq.~\eqref{eq:last-step} implies that $T_2 = 0$.
\end{proof}

\begin{proposition}\label{p:one-atom-comp-filtral}
Let $\K$ be a filtral category with enough subobjects. If its terminal object is an atom, then $\K$ is compatibly filtral.
\end{proposition}

\begin{proof}
Suppose $\one$ is an atom. In view of Lemma~\ref{l:comp-filtral-rephrasing}, in order to show that $\K$ is compatibly filtral it suffices to prove that, for all filtral objects $S_1,S_2\in \K$, the canonical map
\[
p\colon \F(S_1 \times S_2) \to  \F(S_1) \times \F(S_2)
\]
is a localic surjection.

Since $\CHLoc$ is closed under limits in the category of locales (see Section~\ref{s:frames-of-ideals}), $p^*$ is the unique frame homomorphism making the following diagram commute,
  \[\begin{tikzcd}[row sep = 2.5em]
      \Sub{S_1}^\op \arrow{r}{u_1} \arrow{dr}[swap]{\pi_1^{-1}} & \Sub{S_1}^\op + \Sub{S_2}^\op
      & \Sub{S_2}^\op \arrow{l}[swap]{u_2} \arrow{dl}{\pi_2^{-1}} \\
      {} & \Sub{S_1\times S_2}^\op 
      \arrow[<-,dashed]{u}[description]{p^*} &
    \end{tikzcd}\]
    where the top row is a coproduct diagram in $\Frm$.
   By Lemma~\ref{l:dense-implies-surjective}, $p$ is a localic surjection provided it is dense, i.e.\ $p_*(0)=0$.

Now, every element of the frame $\Sub{S_1}^\op + \Sub{S_2}^\op$ is a join of generators of the form $T_1\oplus T_2 = u_1(T_1) \wedge u_2(T_2)$ with $T_1\in \Sub{S_1}^\op$ and $T_2\in \Sub{S_2}^\op$ (see e.g.\ \cite[IV.5.2]{pp}). Note that, in the frame $\Sub{S_1\times S_2}^\op$, we have
\begin{equation}\label{eq:alphastar-generators}
p^*(T_1\oplus T_2) = p^*(u_1(T_1) \wedge u_2(T_2)) = \pi_1^{-1}(T_1) \swedge \pi_2^{-1}(T_2). 
\end{equation}
Hence, in the frame $\Sub{S_1}^\op + \Sub{S_2}^\op$,
\begin{align*}
p_*(0) &= \bigvee{\{T_1\oplus T_2 \mid T_i\in \Sub{S_i}^\op, \ T_1\oplus T_2 \leq p_*(0)\}} \\ 
&= \bigvee{\{T_1\oplus T_2 \mid T_i\in \Sub{S_i}^\op, \ p^*(T_1\oplus T_2) \sleq 0\}} \tag*{$p^*\dashv p_*$}\\ 
&= \bigvee{\{T_1\oplus T_2 \mid T_i\in \Sub{S_i}^\op, \ \pi_1^{-1}(T_1) \swedge \pi_2^{-1}(T_2) = 0\}}. \tag*{eq.~\eqref{eq:alphastar-generators}}
\end{align*}
For each $T_i\in \Sub{S_i}^\op$, consider the inhabited set
\[
\mathcal{U}_{T_i} \coloneqq \{U \in \Sub{S_i}^\op\mid U \sleq T_i \text{ and $U$ is complemented}\}.
\]
As each $S_i$ is a filtral object, $\F(S_i)$ is a Stone locale, and so $T_i = \displaystyle\bigsvee{\mathcal{U}_{T_i}}$. Using the fact that $\oplus$ distributes over suprema (see again \cite[IV.5.2]{pp}), we get
\begin{align*}
& \bigvee{\{T_1\oplus T_2 \mid T_i\in \Sub{S_i}^\op, \ \pi_1^{-1}(T_1) \swedge \pi_2^{-1}(T_2) = 0\}} \\
= & \bigvee{\{U_1\oplus U_2 \mid T_i\in \Sub{S_i}^\op, \ U_i\in \ \mathcal{U}_{T_i}, \ \pi_1^{-1}(\bigsvee{\mathcal{U}_{T_1}}) \swedge \pi_2^{-1}(\bigsvee{\mathcal{U}_{T_2}}) = 0\}} \\
= & \bigvee{\{U_1\oplus U_2 \mid T_i\in \Sub{S_i}^\op, \ U_i\in \ \mathcal{U}_{T_i}, \ \bigsvee{\{\pi_1^{-1}(U_1) \swedge \pi_2^{-1}(U_2)} \mid U_i\in \ \mathcal{U}_{T_i}\} = 0\}} \\
\leq & \bigvee{\{U_1\oplus U_2 \mid U_i\in \Sub{S_i}^\op, \ U_i \text{ is complemented}, \ \pi_1^{-1}(U_1) \swedge \pi_2^{-1}(U_2) = 0\}} \\
\leq & \bigvee{\{U_1\oplus U_2 \mid U_i\in \Sub{S_i}^\op, \ U_1= 0 \text{ or } U_2=0\}} \tag*{Lemma~\ref{l:products-lemma-rectangles}\ref{rectangles-simplified}}
\end{align*}
which is equal to $0$ because $U_1\oplus U_2 =0$ provided that at least one of $U_1$ and $U_2$ is $0$. Thus $p_*(0)=0$, which implies that $\K$ is compatibly filtral.
\end{proof}

\begin{remark}\label{rem:existence-K-em}
The existence of a non-trivial filtral category with enough subobjects such that its terminal object is an atom implies the law of excluded middle. 

To see this, note that if $\K$ is a filtral category satisfying the conditions above, then it is compatibly filtral by Proposition~\ref{p:one-atom-comp-filtral}, and so $\Sub{\one}^\op\cong\Omega$ by Proposition~\ref{p:preservation-fin-lim}. Since~$\one$ is an atom of $\K$, it follows that $\Omega$ is a two-element frame, and in particular a Boolean algebra. Thus, the law of excluded middle holds.
\end{remark}

The next result is an immediate consequence of Theorem~\ref{th:spatial-locales-covered} and Proposition~\ref{p:one-atom-comp-filtral}. Note that its assumptions entail the law of excluded middle (cf.\ Lemma~\ref{l:enough-points-lem} and Remark~\ref{rem:existence-K-em}).
\begin{corollary}\label{cor:main-boolean}
Let $\K$ be a non-trivial filtral pretopos with enough subobjects such that its terminal object is an atom and admits all set-indexed copowers. If compact Hausdorff locales have enough points, the functor $\F\colon\K\to\CHLoc$ is a weak equivalence.
\end{corollary}

Classically, assuming the existence of enough points for compact Hausdorff locales, the category $\CHLoc$ is equivalent to the category $\KHaus$ of compact Hausdorff spaces and continuous maps; cf.\ e.g.\ \cite[\S III.1.10]{jo}. Thus, the characterisation of $\KHaus$ (up to \emph{weak} equivalence) in Theorem~\ref{th:mr-main} can be deduced from Corollary~\ref{cor:main-boolean}. Just observe that every filtral pretopos satisfying the assumptions of Theorem~\ref{th:mr-main} has enough subobjects (since points are special types of subobjects) and its terminal object is an atom; for the latter statement, cf.\ \cite[Lemmas~3.2 and~5.7]{mr}.

\addtocontents{toc}{\protect\setcounter{tocdepth}{1}}
\subsection*{Acknowledgements} 
The authors are grateful to Graham Manuell and Christopher Townsend for further information on the constructive theory of locales and the constructive prime ideal theorem, respectively.
P.~Karazeris and  K.~Tsamis are indebted to Vassilis Aravantinos - Sotiropoulos for useful discussions. 
\addtocontents{toc}{\protect\setcounter{tocdepth}{2}}

\bibliographystyle{plain}

\end{document}